\documentclass[12pt, reqno]{amsart}
\usepackage{amsmath}
\usepackage{amsthm}
\usepackage{amssymb}
\usepackage{graphicx}
\usepackage{indentfirst}
\usepackage{bm}
\usepackage{xcolor}
\usepackage[pdfauthor={},
pdftitle={},%
pdftex]{hyperref}
\hypersetup{
	colorlinks,
	citecolor=red,
	filecolor=black,
	linkcolor=blue,
	urlcolor=blue
}
\usepackage{caption}
\usepackage{subcaption}
\usepackage{float}
\usepackage{forest}
\usepackage{tikz-cd}
\usepackage[all,cmtip]{xy}
\usepackage[mathscr]{euscript}
\allowdisplaybreaks
\makeatother

\usepackage[left=3.00cm, right=3.00cm, top=2.00cm, bottom=3.00cm]{geometry} 
\newcommand{\R}{\mathbb{R}}
\newcommand{\N}{\mathbb{N}}

\newcommand{\compcent}[1]{\vcenter{\hbox{$#1\circ$}}}
\newcommand{\comp}{\mathbin{\mathchoice
		{\compcent\scriptstyle}{\compcent\scriptstyle}
		{\compcent\scriptscriptstyle}{\compcent\scriptscriptstyle}}}

\newenvironment{tree}[1]
{
	\begin{forest}
		for tree={minimum width = 4em,
			delay = {where content={}{shape=coordinate}{},
				calign=fixed edge angles,
				calign angle=45,      
				grow=north,
			},
			text centered,
		}
		#1
	\end{forest} 
}

\newtheoremstyle{ojmSty}{3pt}{3pt} {}{}{\scshape}{.}{.5em}{}

\theoremstyle{plain}
\newtheorem{thm}{Theorem}[section]
\newtheorem{lemma}[thm]{Lemma}
\newtheorem{innercustomgeneric}{\customgenericname}
\providecommand{\customgenericname}{}
\newcommand{\newcustomtheorem}[2]{
	\newenvironment{#1}[1]
	{
		\renewcommand\customgenericname{#2}
		\renewcommand\theinnercustomgeneric{##1}
		\innercustomgeneric
	}
	{\endinnercustomgeneric}
}
\newcustomtheorem{customthm}{Theorem}
\newcustomtheorem{customobs}{Observation}
\newtheorem*{thm*}{Theorem}

\newtheorem{prop}[thm]{Proposition}
\theoremstyle{ojmSty}
\newtheorem{defi}[thm]{Definition}
\newtheorem{remark}{Remark}[section]

\begin{document}
	
	\title{Equivalence Between Four Models of Associahedra}
	\author{Somnath Basu, Sandip Samanta}
	\address{Department of Mathematics and Statistics, Indian Institute of Science Education and Research Kolkata,
		Mohanpur, West Bengal, India -- 741 246}
	\email{somnath.basu@iiserkol.ac.in, ss18ip021@iiserkol.ac.in}
	
	\begin{abstract}
		We present a combinatorial isomorphism between Stasheff associahedra and an inductive cone construction of those complexes given by Loday. We give an alternate description of certain polytopes, known as multiplihedra, which arise in the study of $A_\infty$ maps. We also provide new combinatorial isomorphisms between Stasheff associahedra, collapsed multiplihedra, and graph cubeahedra for path graphs. 
	\end{abstract}

	\maketitle
	\tableofcontents
	\section{Introduction}
    Dov Tamari, in his 1951 thesis \cite{D.Tamari}, first described associahedra (with notation $\mathscr{M}_{n-1}$) as the realization of his poset lattice of bracketings (parenthesizations) of a word with $n$ letters. He had also pictured the $1$, $2$, and $3$ dimensional cases (cf. Figure \ref{fig:D.Tamari}). Later these were rediscovered by Jim Stasheff \cite{HAH} in his 1960 thesis on homotopy associativity and based loop spaces. Stasheff had defined these (with notation $K_n$) as a convex, curvilinear subset of the $(n-2)$ dimensional unit cube (cf. Figure \ref{fig:Stasheff}) such that it is homeomorphic to the cube. Convex polytope realizations of associahedra were subsequently done by many people \cite{HuguetTamari, Haiman, Lee, Loday1}. These polytopes are commonly known as \textit{associahedra} or \textit{Stasheff complexes}.  \\
	\hspace*{0.5cm}Ever since Stasheff's work, associahedra (and their face complexes) have continued to appear in various mathematical fields apart from its crucial role in homotopy associative algebras and its important role in discrete geometry. Indeed, the associahedron $K_{n-1}$ appears as a fundamental tile of $\overline{\mathcal{M}}_{0,n}(\mathbb{R})$, the compactification of the real moduli space of punctured Riemann sphere \cite{Devadoss3}. It also appears in the analysis of the compactified moduli space of \textit{nodal disks} with markings, as described by Fukaya and Oh \cite{Fukaya}. An important connection between associahedra (and its generalizations) and finite root systems was established in 2003 by the work of Fomin and Zelevinsky \cite{FominZelevinsky}. In 2006 Carr and Devadoss \cite{Devadoss2} generalized associahedra to graph associahedra $\mathcal{K}G$ for a given graph $G$. These appear as the tiling of minimal blow-ups of certain Coxeter complexes \cite{Devadoss2}. In particular, if $G$ is a path graph, then $\mathcal{K}G$ is an associahedron. Bowlin and Brin \cite{BowlinBrin}, in 2013, gave a precise conjecture about existence of coloured paths in associahedra. They showed that this conjecture is equivalent to the four colour theorem (4CT). Earlier, in 1988, there was a celebrated work \cite{SleatorTarjanThurston} of Sleator, Tarjan and Thurston on the diameter of associahedra. While working on dynamic optimality conjecture, they had used hyperbolic geometry techniques to show that the diameter of $K_d$ is at most $2d-8$ when $d\geq 11$, and this bound is sharp when $d$ is large enough. Pournin \cite{Pournin}, almost twenty five years later, showed that this bound is sharp for $d\geq 11$. Moreover, his proof was combinatorial. Even in theoretical physics, recent works \cite{Mizera,HamedBaiHeYan,FerroTomasz} indicate that associahedron plays a key role in the theory of scattering amplitudes.
	\begin{figure}[htbp]
		\centering
		\begin{subfigure}[b]{0.45\linewidth}
			\centering
			\includegraphics[width=0.9\linewidth]{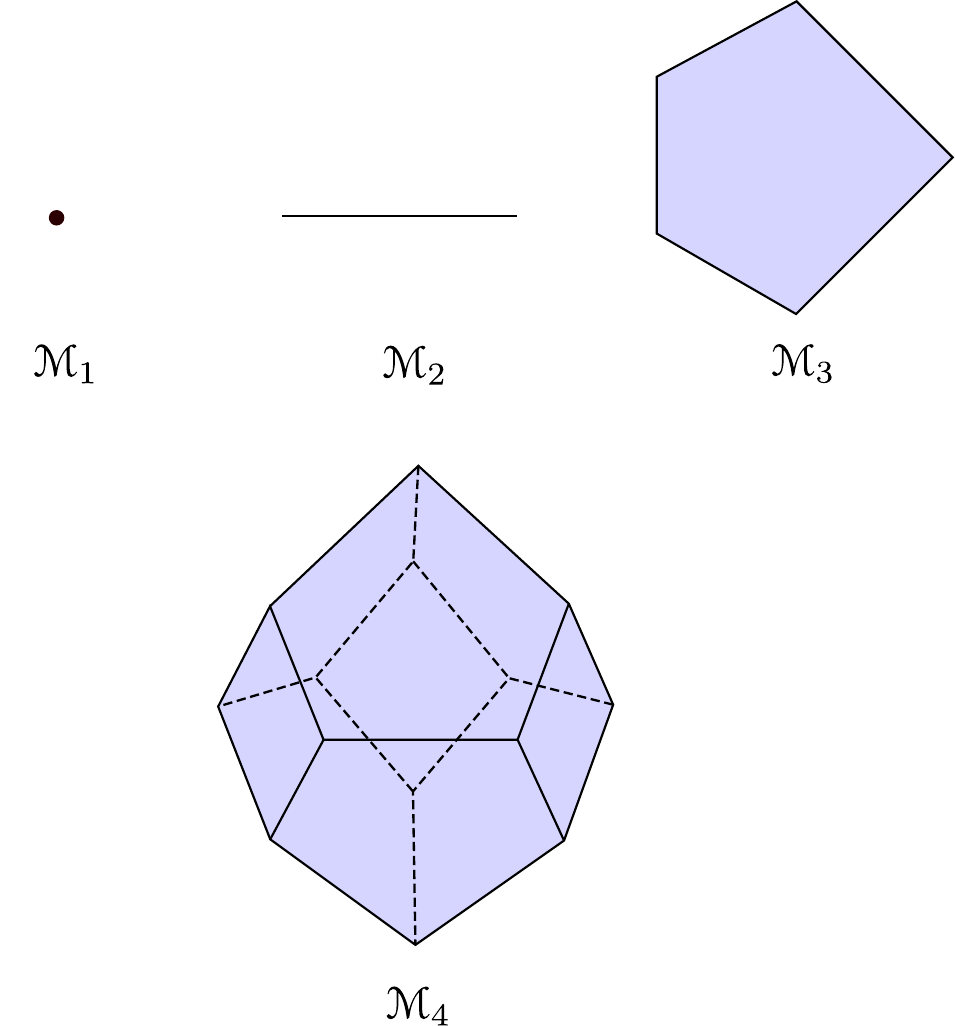}
			\caption{Tamari's associahedra}
			\label{fig:D.Tamari}
		\end{subfigure}
		\hspace{1cm}
		\begin{subfigure}[b]{0.45\linewidth}
			\centering
			\includegraphics[width=0.8\linewidth]{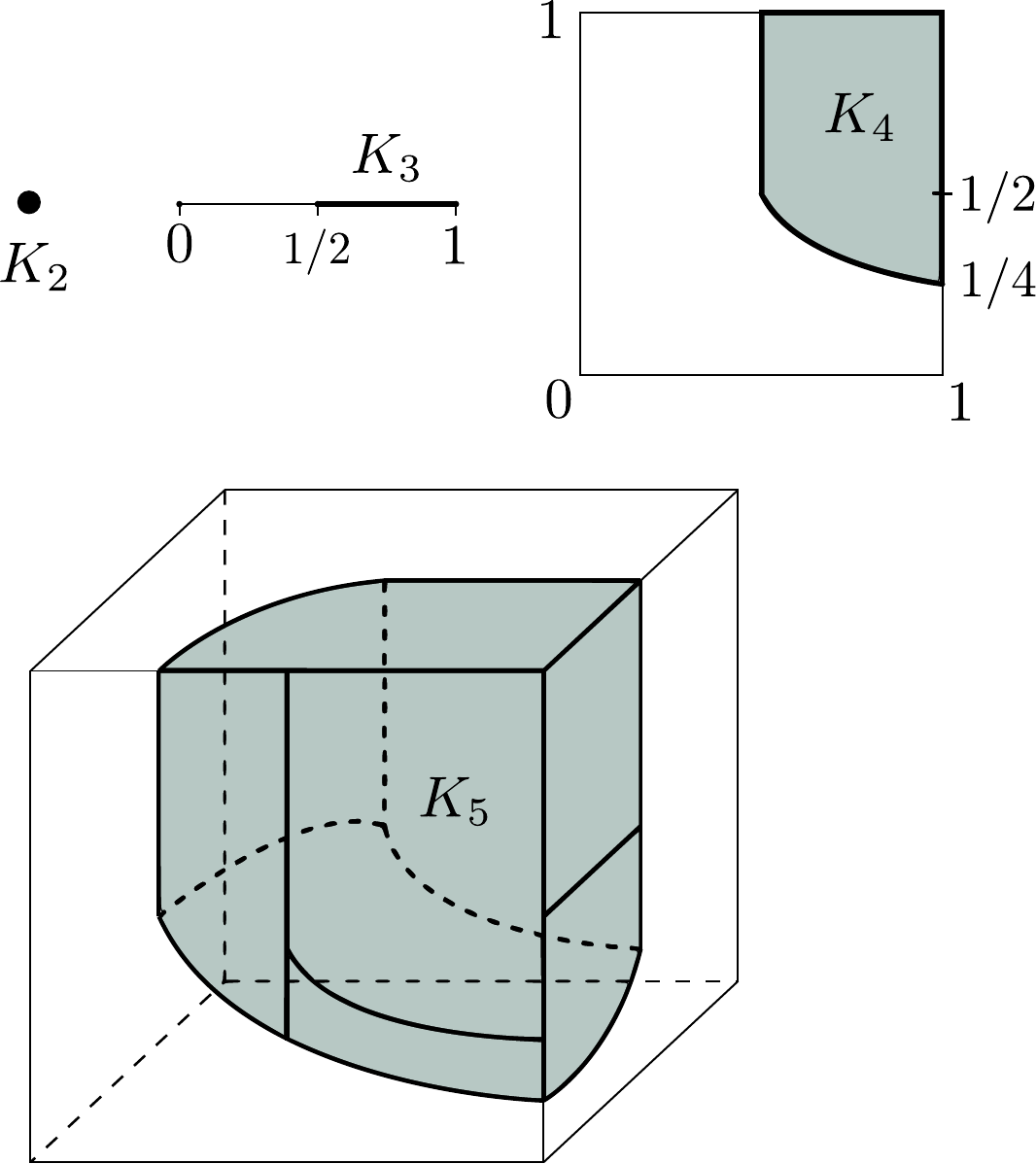}
			\caption{Stasheff's associahedra}
			\label{fig:Stasheff}
		\end{subfigure}
		\caption{Earliest realizations of associahedra}
	\end{figure}
	
	Let us briefly recall the construction in \cite{HAH}. Stasheff, respecting Tamari's description, had sub-divided the boundary of $K_n$ in such a way that the number of faces of codimension $1$ and the adjacencies in his model matched with that in \cite{D.Tamari}. The boundary of $K_n$, denoted by $L_n$, is the union of homeomorphic images of $K_p\times_r K_q$ ($p+q=n+1$, $r=1,2,...,p$), where $K_p\times_r K_q$ corresponds to the bracketing $x_1\ldots (x_r\ldots x_{r+q-1})\ldots x_n$. 
	Stasheff started with $K_2$ as a point and defined $K_n$, inductively, as a cone over $L_n$. This definition of $K_n$ involves $K_2$ through $K_{n-1}$ all together. \\
	\hspace*{0.5cm}As associahedra are contractible, these are of less interest as spaces in isolation. However, as combinatorial objects, the key properties of it are inherent in its description as a convex polytope. Much later, in 2005, J. L. Loday \cite{Loday2} gave a different inductive construction of $K_n$ starting from $K_{n-1}$, leaving it to the reader to verify the details. Being a predominantly topological construction, it is not apparent why the cone construction of Loday gives rise to the known combinatorial structure on the associahedra. It is, therefore, natural to search for an explicit combinatorial isomorphism between these two constructions, leading to our first result (Theorem \ref{thm:StasheffLoday}).
	\begin{customthm}{A} \label{thm:iStasheffLoday}
		Stasheff complexes are combinatorially isomorphic to Loday's cone construction of associahedra.
	\end{customthm}
	There is another set of complexes $\mathcal{J}(n)$, known as \textit{multiplihedra}, which were first introduced and pictured by Stasheff \cite{Stas} in order to define $A_\infty$ maps between $A_\infty$ spaces, for $n\leq 4$. Mau and Woodward \cite{Woodward} have shown $\mathcal{J}(n)$'s to be compactifications of the moduli space of quilted disks. Boardman and Vogt \cite{Boardman} defined $\mathcal{J}(n)$ in terms of painted trees (refer to Definition \ref{def:PaintedTree}). The first detailed definition of $\mathcal{J}(n)$ and its combinatorial properties were described by Iwase and Mimura \cite{IwaseMimura}, while its realization as convex polytopes was achieved by Forcey \cite{Forcey1}, combining the description of Boardman-Vogt and Iwase-Mimura. Later, Devadoss and Forcey \cite{Devadoss4} generalized multiplihedra to \textit{graph multiplihedra} $\mathcal{J}G$ for a given graph $G$. \\
 
	\hspace*{0.5cm}In the study of $A_\infty$ maps from an $A_\infty$ space to a strictly associative $H$ space (i.e., a topological monoid), multiplihedra degenerate to what we call \textit{collapsed multiplihedra}. Stasheff \cite{Stas} had pointed out that these complexes resemble associahedra. It has been observed that collapsed multiplihedra can be viewed as degeneration of graph multiplihedra for path graphs.
	It was long assumed that for $A_\infty$ maps from a strictly associative $H$ space to a $A_\infty$ space, multiplihedra would likewise degenerate to yield associahedra. But it was Forcey \cite{Forcey2} who realized that new polytopes were needed. These were constructed by him and named \textit{composihedra}.  \\
	\hspace*{0.5cm}In this paper, we will give an equivalent definition (Definition \ref{def:Multiplihedra}) of multiplihedra, which induces a definition for collapsed multiplihedra (Definition \ref{def:CollMulti}). Using this definition, we will give a proof of the following (Proposition \ref{thm:StasMulti}) by providing a new bijection of underlying posets.
	\begin{customobs}{a} \label{thm:iStasMulti}
		Stasheff complexes and collapsed multiplihedra are combinatorially isomorphic.
	\end{customobs}
	\noindent There is a well-known bijection $bij_3$ (cf. Forcey's paper \cite[p. 195]{ForceyLattice}; prior to Remark 2.6 and Figure 7) which is different from ours. 
	
	In 2010, Devadoss, Heath, and Vipismakul \cite{Devadoss1} defined a polytope called \textit{graph cubeahedron} (denoted by $\mathcal{C}G$) associated to a graph $G$. These are obtained by truncating certain faces of a cube. They gave a convex realization of these polytopes as simple convex polytopes whose face poset is isomorphic to the poset of design tubings for graphs. Graph cubeahedra for cycle graphs $G$ (called \textit{halohedra}) appear as the moduli space of annulus with marked points on one boundary circle. In this paper, we are mainly interested in $\mathcal{C}G$ for path graphs $G$ and will prove the following (Proposition \ref{thm:MultiCubea}) by providing a new bijection of underlying posets.
	\begin{customobs}{b}\label{thm:iMultiCubea}
		The collapsed multiplihedra and graph cubeahedra for path graphs are combinatorially isomorphic.
	\end{customobs}
	\noindent It turns out that bijection obtained between the posets governing Stasheff complexes and graph cubeahedra (for path graphs), by combining our bijections from Observations \ref{thm:iStasMulti} and \ref{thm:iMultiCubea}, is the bijection of posets defined in \cite[Proposition 14]{Devadoss1}. From our perspective, the bijection in Observation \ref{thm:iMultiCubea} is natural. Combining Theorem \ref{thm:iStasheffLoday}, Observations \ref{thm:iStasMulti} and \ref{thm:iMultiCubea}, we obtain the following result (Theorem \ref{thm:Main}).
	 
	\begin{customthm}{B}\label{thm:iMain}
		The four models of associahedra - Stasheff complexes, complexes obtained by Loday's cone construction, collapsed multiplihedra, graph cubeahedra for path graphs - are all combinatorially isomorphic.
	\end{customthm}
	
	\noindent \textsc{Organization of the paper}. The paper is organized as follows. In \S \ref{Stasheff}, we will review some of the definitions and results related to Stasheff's description of associahedra. In \S \ref{LodayCone}, the description of Loday's cone construction and some related theorems are presented while in \S \ref{CollapsedMulti} an equivalent definition of multiplihedra and collapsed multiplihedra is given. In \S \ref{Cubeahedra} the definition of tubings, design tubings, graph cubeahedra, and related results are presented. The next section \S \ref{Main} contains the proof of the main result (Theorem \ref{thm:iMain}), which is a combination of three results. In \S \ref{LodayStasheff} we prove Theorem \ref{thm:iStasheffLoday} while \S \ref{StasMulti} and \S \ref{MultiCubea} are devoted to the proofs of Observations \ref{thm:iStasMulti} and \ref{thm:iMultiCubea} respectively.\\
	
	\noindent \textsc{Acknowledgments.} The authors would like to thank Stefan Forcey for an initial discussion on this topic as well as several useful comments on the first draft. The first author acknowledges the support of SERB MATRICS grant MTR/2017/000807 for the funding opportunity. The second author is supported by a PMRF fellowship. We would like to express our sincere gratitude to the anonymous referee whose insights and constructive comments helped in improving this manuscript.

	\section{Description of Four Models of Associahedra}
	An $H$-space is a topological space $X$ equipped with a binary operation $m: X^2 \to X$ having a unit $e$. It is a natural generalization of the notion of topological groups. We can rewrite $m$ as a map $m_2:K_2\times X^2\to X$, where $K_2$ is a point. If $m$ is not associative but homotopy associative (called weakly associative), then we have a map $m_3:K_3\times X^3\to X$ defined through the homotopy between $m\comp (m\times 1)$ and $m\comp (1\times m)$, where $K_3$ is an interval. Similarly, we can define five different maps from $X^4 \to X$ using $m$, and between any two such maps, there are two different homotopies (using the chosen homotopy associativity). If those two homotopies are homotopic, this defines a map $m_4:K_4\times X^4\to X$, where $K_4$ is a filled pentagon. If we continue this process,  we obtain a map $m_n:K_n \times X^n\to X$ for $n\geq 2$. These complexes $K_n$, called associahedra, are our main objects of interest.\\

	\hspace*{0.5cm}We will briefly describe the four models of associahedra, one in each subsection, we are concerned with: Stasheff complexes, Loday's cone construction, collapsed multiplihedra, and graph cubeahedra for path graphs.
	\subsection{Stasheff complexes} \label{Stasheff}
	Stasheff defined for each $i\geq 2$, a special cell complex $K_i$ as a subset of $I^{i-2}$. It is a simplicial complex and has $i$ degeneracy operators $s_1,...,s_i$. Moreover, $K_i$ has $\binom{i}{2}-1$ faces of codimension $1$. The complexes $K_i$, as combinatorial objects, are more complicated than the standard simplices $\Delta^{i-2}$. According to Stasheff \cite{HAH}, it is defined through following intuitive content:\\
	\indent Consider a word with $i$ letters, say $x_1x_2\ldots x_i$, and all meaningful ways of inserting one pair of parentheses. To each such insertion except for $(x_1x_2...x_i)$, there corresponds a cell of $L_i$, the boundary of $K_i$. If the parentheses enclose $x_k$ through $x_{k+s-1}$, we regard this cell as $K_r\times_k K_s$, the homeomorphic image of $K_r\times K_s$ under a map which we call $\partial_k(r,s)$, where $r+s=i+1$. Two such cells intersect only on their boundaries and the `edges' so formed correspond to inserting two pairs of parentheses in the word. 
    Specifically, the intersection of two cells, namely $K_r\times_k K_s$ and $K_{r'}\times_{k'} K_{s'}$, occurs if and only if one of the following conditions is satisfied: $(k\leq k'<k'+s'-1\leq k+s-1)$, or $(k'\leq k<k+s-1\leq k'+s'-1)$, or $(k+s-1<k')$, or $(k'+s'-1<k)$. Furthermore, if these two cells intersect, the intersection takes place along an $(i-4)$-dimensional sub-cell, which is either of the form $K_a \times_j (K_b \times_l K_c)$ appears as a cell of $K_a\times_j \partial K_{b+c-1}$, or of the form $(K_a\times_j K_b)\times_l K_c$ appears as a cell of $\partial K_{a+b-1}\times_l K_c$ for $a+b+c=i+2$. 
	Now the adjacency criterion is given by the following relations: 
	\begin{itemize}
		\item[(a)] $\partial_{j}(r, s+t-1)\left(1 \times \partial_{k}(s, t)\right)=\partial_{j+k-1}(r+s-1, t)\left(\partial_{j}(r, s) \times 1\right)$
		\item[(b)] $\partial_{j+s-1}(r+s-1, t)\left(\partial_{k}(r, s) \times 1\right)=\partial_{k}(r+t-1, s)\left(\partial_{j}(r, t) \times 1\right)(1 \times T)$
	\end{itemize}
	where $T: K_{s} \times K_{t} \rightarrow K_{t} \times K_{s}$ permutes the factors. In terms of homeomorphic images of $K_r\times K_s\times K_t$, the above two relations are equivalent respectively to the identifications
	\begin{align}
		&K_r\times_j(K_s\times_k K_t)=(K_r\times_j K_s)\times_{j+k-1} K_t\label{eq:1}\\
		&(K_r\times_k K_s)\times_{j+s-1} K_t=(K_r\times_j K_t)\times_k K_s\label{eq:2}
	\end{align} 
    One can easily track these relations once they are identified with $2$-bracketing of the word $x_1x_2\ldots x_i$. The cell $K_a\times_j(K_b\times_l K_c)$ corresponds to the nested $2$-bracketing $$x_1\ldots (x_j\ldots (x_{j+l-1}\ldots x_{j+l+c-2})\ldots x_{j+b+c-2})\ldots x_i,$$ for $a+b+c=i+2$. The cell $(K_a \times_j K_b)\times_l K_c$ corresponds to $y_1\ldots(y_j\ldots y_{j+b-1})\ldots y_{a+b-1}$ with 
    $$y_u=\begin{cases}
        x_u &\text{ if }u<l\\
        (x_l\ldots x_{l+c-1}) &\text{ if }u=l\\
        x_{u+c-1} &\text{ if }u>l
    \end{cases}$$
    Now if $j\leq l\leq j+b$, then $(K_a \times_j K_b)\times_l K_c$ represents a nested $2$-bracketing and thus corresponds to a cell of first type; this case is reflected in the identification (\ref{eq:1}) above. If $l+c<j$ or $j+b<l$, then $(K_a \times_j K_b)\times_l K_c$ represents a disjoint $2$-bracketing and that corresponds to two possible cells of the second type; this case is reflected in (\ref{eq:2}).
    
	This is enough to obtain $K_{i}$ by induction. Start with $K_{2}=\{\ast\}$ as a point. Given $K_{2}$ through $K_{i-1}$, construct $L_{i}$ by fitting together copies of $K_{r} \times_k K_{s}$ as indicated by the above conditions, and take $K_{i}$ to be the cone on $L_{i}$.
	Stasheff proved that these complexes are homeomorphic to cubes.
	\begin{prop}{\cite[Proposition 3]{HAH}}
		$K_{i}$ is homeomorphic to $I^{i-2}$ and  
		degeneracy maps $s_{j}: K_{i} \rightarrow K_{i-1}$ for $1 \leq j \leq i$ can be defined so that the following relations hold:
		\begin{enumerate}
			\item $s_{j} s_{k}=s_{k} s_{j+1}$ for $k \leq j$.
			\item $s_{j} \partial_{k}(r, s)=\partial_{k-1}(r-1, s)\left(s_{j} \times 1\right)$ for $j<k$ and $r>2$.
			\item $s_{j} \partial_{k}(r, s)=\partial_{k}(r, s-1)\left(1 \times s_{j-k+1}\right)$ for $s>2, k \leq j<k+s,\\
			s_{j} \partial_{k}(i-1,2)=\pi_{1}$ for $1<j=k<i$ and $1<j=k+1 \leq i, $\\
			$s_{1} \partial_{2}(2, i-1)=\pi_{2}$ and $s_{i} \partial_{1}(2, i-1)=\pi_{2}$,\\
			where $\pi_{m}$ for $m=1,2$ is projection onto the $m$-th factor.
			\item $s_{j} \partial_{k}(r, s)=\partial_{k}(r-1, s)\left(s_{j-s+1} \times 1\right)$ for $k+s \leq j$.
		\end{enumerate}
	\end{prop}
	
	\noindent
	Using boundary maps $\partial_k(r,s)$ and degeneracy maps $s_j$, Stasheff defined the following.
	\begin{defi}[$A_n$ form and $A_n$ space]
		An $A_n$ form on a space $X$ consists of a family of maps $m_i:K_i\times X^i\to X$ for $2\leq i\leq n$ such that
		\begin{enumerate}
			\item there exists $e\in X$ with $m_{2}(*, e, x)=m_{2}(*, x, e)=x$ for $x \in X, *=K_{2}$.
			\item For $\rho \in K_{r}, \sigma \in K_{s}, r+s=i+1$, we have
			$$
			\begin{aligned}
				m_{i}\left(\partial_{k}(r, s)(\rho, \sigma), x_{1}, \cdots, x_{i}\right)= 
				m_{r}\left(\rho, x_{1}, \cdots, x_{k-1}, m_{s}\left(\sigma, x_{k}, \cdots, x_{k+s-1}\right), x_{k+s}, \cdots, x_{i}\right).
			\end{aligned}
			$$
			\item For $\tau \in K_{i}$ and $i>2$, we have
			$$
			m_{i}\left(\tau, x_{1}, \cdots, x_{j-1}, e, x_{j+1}, \cdots, x_{i}\right)=m_{i-1}\left(s_{j}(\tau), x_{1}, \cdots, x_{j-1}, x_{j+1}, \cdots, x_{i}\right).
			$$
		\end{enumerate}
		The pair $(X,\{m_i\}_{2\leq i\leq n})$ is called an $A_n$ space. \\
		If the maps $m_i$ exist for all $i$, then it is called an $A_\infty$ form, and the corresponding pair is called an $A_\infty$ space.
	\end{defi}
	Homotopy associative algebras (or $A_\infty$ algebras), $A_\infty$ spaces, and operads have been extensively studied. The interested reader is directed to the excellent books \cite{May, Boardman, Adams} and introductory notes \cite{Keller}.
	
	The correspondence between the faces of Stasheff complexes (associahedra) and the bracketings indicate that these complexes can also be defined as follows.
	\begin{defi}[Associahedron]\label{def:Associahedra1}
		Let $\mathfrak{P}(n)$ be the poset of bracketings of a word with $n$ letters, ordered such that $p < p^{\prime}$ if $p$ is obtained from $p^{\prime}$ by adding new brackets. The associahedron $K_{n}$ is a convex polytope of dimension $n-2$ whose face poset is isomorphic to $\mathfrak{P}(n)$.
	\end{defi}
	\noindent
	This construction of the polytope $K_{n}$ was first given in 1984 by Haiman in his (unpublished) manuscript \cite{Haiman}. In 1989, C. Lee \cite[Theorem 1]{Lee} proved this by considering the collection of all sets of mutually non-crossing diagonals of a polygon. 
	Observe that the sets of mutually non-crossing diagonals of an $(n+1)$-gon are in bijective correspondence with the bracketings of a word with $n$ letters.
	We will use this description later in \S \ref{StasMulti}.

	\subsection{Loday's Cone Construction}\label{LodayCone}
	
	From the combinatorial description given by Stasheff, the associahedron $K_n$ is a complex of dimension $n-2$ whose vertices are in bijective correspondence with the $(n-2)$-bracketing of the word $x_1x_2\ldots x_n$.
	But each $(n-2)$-bracketing of the word $x_1x_2\ldots x_n$ corresponds to a rooted planar binary tree with $n+1$ leaves, one of them being the root. For example, the planar rooted trees associated with $x_1(x_2(x_3x_4))$ and  $(x_1x_2)(x_3x_4)$ are depicted below (cf. Figure \ref{fig:tree1}, \ref{fig:tree2}), the root being represented by the vertical leaf in each case. 
	
	\begin{figure}[H]
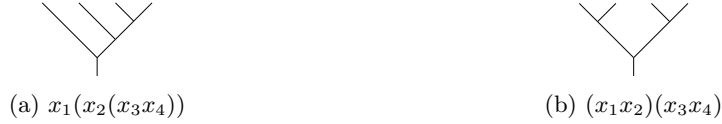

		\centering
		\begin{subfigure}[b]{0.45\linewidth}
			\centering
			\scalebox{0.25}{\tree{[[[[[][,tier=L]][,tier=L]][,tier=L]]]}}
			\subcaption{$x_1(x_2(x_3x_4))$}
			\label{fig:tree1}
		\end{subfigure}
		\begin{subfigure}[b]{0.45\linewidth}
			\centering
			\scalebox{0.25}{\tree{[[[[][,tier=L]][~,phantom,fit=band][[][,tier=L]]]]}}
			\subcaption{$(x_1x_2)(x_3x_4)$}
			\label{fig:tree2}
		\end{subfigure}
		\caption{Correspondence between bracketing and rooted binary tree}
	\end{figure}
	
	\noindent Thus $K_n$ can also be thought of as a complex of dimension $n-2$ whose vertices are in bijective correspondence with planar rooted binary trees with $n$ leaves and $1$ root.
	Let $Y_{n}$ be the set of such trees. The trees are depicted below for $2\leq n\leq 4$.\\
	$$
	Y_{2}=
	\left\{
	\begin{matrix}
		\scalebox{0.3}{\tree{[[[][]]]}}
	\end{matrix}
	\right\},\
	Y_{3}=
	\left\{
	\begin{matrix}
		\scalebox{0.2}{\tree{[[[[][,tier=L]][,tier=L]]]}}, \scalebox{0.2}{\tree{[[[,tier=L][[,tier=L][]]]]}}
	\end{matrix}
	\right\},\
	Y_{4}=
	\left\{
	\begin{matrix}
		\scalebox{0.15}{\tree{[[[,tier=L][[,tier=L][[,tier=L][]]]]]}}, \scalebox{0.20}{\tree{[[[[][]][[][]]]]}}, \scalebox{0.15}{\tree{[[[,tier=L][[[][,tier=L]][,tier=L]]]]}}, \scalebox{0.15}{\tree{[[[[,tier=L][[,tier=L][,tier=L]]][,tier=L]]]}}, \scalebox{0.15}{\tree{[[[[[][,tier=L]][,tier=L]][,tier=L]]]}}
	\end{matrix}
	\right\}
	$$
	\vspace{2pt}
	
	Any $t \in Y_{n}$ is defined to have \textit{degree} $n$. We label the leaves (not the root) of $t$ from left to right by $0, 1, \cdots, n-1$. Then we label the internal vertices by $1,2, \cdots, n-1$.
	The $i$-th internal vertex is the one that falls in between the leaves $i-1$ and $i$. We denote by $a_{i}$, respectively $b_{i}$, the number of leaves on the left side, respectively right side, of the $i$-th vertex.
	The product $a_{i} b_{i}$ is called the \textit{weight} of the $i$-th internal vertex. To each tree $t\in Y_{n}$, we associate the point $M(t) \in \mathbb{R}^{n-1}$, whose $i$-th coordinate is the weight of the $i$-th vertex:
	$$
	M(t)=\left(a_{1} b_{1}, \cdots, a_{i} b_{i}, \cdots, a_{n-1} b_{n-1}\right) \in \mathbb{R}^{n-1}
	$$
	
	\noindent For instance,\vspace*{0.2cm}
	$$\begin{aligned}
		&M\left(
		\begin{matrix}
			\scalebox{0.2}{\tree{[[[][]]]}}
		\end{matrix}
		\right)=(1),\ 
		M\left(
		\begin{matrix}
			\scalebox{0.18}{\tree{[[[[][,tier=L]][,tier=L]]]}}
		\end{matrix}
		\right)=(2,1),\ 
		M\left(
		\begin{matrix}
			\scalebox{0.18}{\tree{[[[,tier=L][[,tier=L][]]]]}}
		\end{matrix}
		\right)=(1,2),\\[5pt]
		&M\left(
		\begin{matrix}
			\scalebox{0.13}{\tree{[[[,tier=L][[,tier=L][[,tier=L][]]]]]}}
		\end{matrix}
		\right)=(1,2,3),\ 
		M\left(
		\begin{matrix}
			\scalebox{0.18}{\tree{[[[[][]][[][]]]]}}
		\end{matrix}
		\right)=(1,4,1)
	\end{aligned}$$
	
	\noindent Observe that the weight of a vertex depends only on the sub-tree that it determines. Using these integral coordinates, Loday \cite{Loday1} gave a convex realization of $K_{n+1}$ in $\mathbb{R}^n$.
	
	\begin{lemma}{\cite[Lemma 2.5]{Loday1}}
		For any tree $t \in Y_{n+1}$ the coordinates of the point $M(t)=\left(x_{1}, \cdots, x_{n}\right) \in \mathbb{R}^{n}$ satisfy the relation $$\sum_{k=1}^{n} x_{k}=\textstyle{\frac{1}{2} n(n+1)} .$$ 
		Thus, it follows that 
		$$M(t) \in H_n=\left\{(x_1,...,x_n)\in \mathbb{R}^n :x_1+x_2+...+x_n= \textstyle{\frac{n(n+1)}{2}}\right\}.$$
	\end{lemma}
	
	\begin{thm}{\cite[Theorem 1.1]{Loday1}}
		The convex hull of the points $M(t) \in \mathbb{R}^{n}$, for $t\in Y_{n+1}$, is a realization of the Stasheff complex $K_{n+1}$ of dimension $n-1$.
	\end{thm}
	\noindent For example, the complex $K_5$ lies in the hyperplane $H_4$ in $\R^4$. Under an isometric transformation of $H_4$ to $\R^3$ (i.e., $x_4=0$ hyperplane), the embedded picture of $K_5$ is shown in Figure \ref{fig:embeddedK5}.
	\begin{figure}[H]
		\centering
		\includegraphics[width=0.3\linewidth]{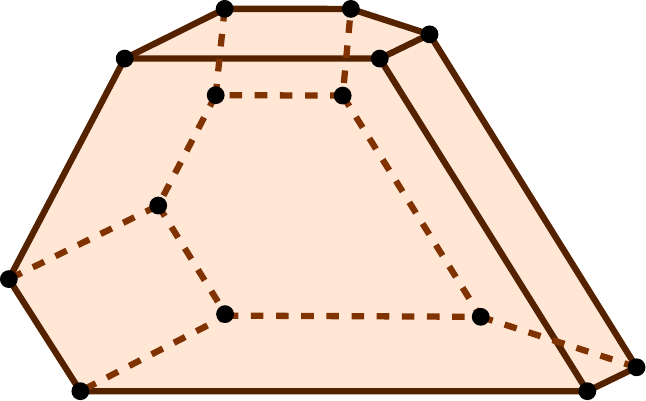}
		\caption{Loday's embedded $K_5$ in $\R^3$}
		\label{fig:embeddedK5}
	\end{figure}
	
	Now starting with $K_2$ as a point, Loday \cite[\S 2.4]{Loday2} gave a different inductive construction of the complexes $K_{n+1}$. The steps are as follows: 
	
    \begin{enumerate}
		\item Start with the associahedron $K_n$, which is a topological ball with the cellular sphere as the boundary. The cells of the boundary are of the form $K_p\times_r K_q$ where $p+q=n+1$ and $r=1,2,...,p$.  
		
		\item Enlarge each cell $K_p\times_r K_q$ of the the boundary of $K_n$ into a cell of dimension $n$ by replacing it with $K_{p+1}\times_r K_q$ keeping the adjacency of the cells intact. Explicitly, suppose two cells $K_p\times_r K_q$ and $K_{p'}\times_{r'} K_{q'}$ are adjacent with a common boundary sub-cell $K_a\times_j(K_b\times_l K_c)$ or $(K_a\times_j K_b)\times_l K_c$ (that are the only possibilities, check identification (\ref{eq:1}), (\ref{eq:2}) in \S\ref{Stasheff}) on the boundary of $K_n$. Then the cells $K_p\times_r K_q$ and $K_{p'}\times_{r'} K_{q'}$ are enlarged to $K_{p+1}\times_r K_q$ and $K_{p'+1}\times_{r'} K_{q'}$ so that they share the common enlarged boundary sub-cell $K_{a+1}\times_j(K_b\times_l K_c)$ or $(K_{a+1}\times_j K_b)\times_l K_c$. We denote the total enlarged complex by $\widehat{K}_n$. 
        Topologically, one may think this process of enlargement as a quotient space as follows.
        \begin{gather*}
            \widehat{K}_n=\Big(K_n\bigsqcup_{(p,q,r)\in V_n} (K_{p+1}\times_r K_q)\Big)\Big/\sim, \text{ where}\\
            V_n=\{(p,q,r)\in \N^3:  p,q\geq 2; p+q=n+1; 1\leq r\leq p\} \text{ and}\\
            \quad \partial(K_n)\ni K_p\times_r K_q\sim  (K_2 \times_1 K_p)\times_r K_q \in \partial(K_{p+1}\times_r K_q)\text{ with identification (\ref{eq:1}), (\ref{eq:2}) on} \\
            \partial(K_{p+1}\times_r K_q)\text{ and } \partial(K_{p'+1}\times_{r'} K_{q'}) \text{ for all }(p,q,r), (p',q',r')\in V_n. \vspace*{-0.2cm}	
        \end{gather*}	
		\item Take the topological cone over the above enlargement $\widehat{K}_n$ and declare that to be $K_{n+1}$, i.e. $K_{n+1}:= C(\widehat{K}_n)=\frac{\widehat{K}_n\times [0,1]}{\widehat{K}_n\times \{1\}}$. By regarding $\big[\widehat{K}_n\times \{1\}\big]$ as the abstract cone point $x_0$ (say), one may think $[(x,t)]$ in $C(\widehat{K}_n)$ as the point $tx_0+(1-t)x$ on the segment joining $x$ to $x_0$ for $t\in[0,1]$.
	\end{enumerate}
    Note that the above construction of $K_{n+1}$ from $K_n$ does not give any convex structure to it. But embedding the enlarged complex $\widehat{K}_n$ in $\R^{n-1}$ and choosing an appropriate cone point there, it is possible to realize $K_{n+1}=C(\widehat{K}_n)$ as a convex polytope in $\R^n$. This is illustrated in the following examples in low dimensions. However, for the general cases, we shall restrict ourselves to the topological part only.
	\begin{itemize}
		\item[(i)] To construct $K_3$ from $K_2$, form the enlarged complex $\widehat{K}_2$, which is a point (as $K_2$ has no boundary). Then $K_3$ is a cone over the point $\widehat{K}_2$, i.e., an interval.
		\begin{figure}[H]
			\centering		
			\includegraphics[width=0.8\linewidth]{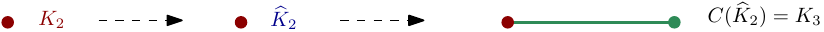}
			\caption{$K_3$ from $K_2$}
		\end{figure}
		\vspace*{-0.2cm}\item[(ii)] To construct $K_4$ from $K_3$, note that $K_3$ has two boundary points namely $K_2\times_1 K_2$ and $K_2\times_2 K_2$. Thus $\widehat{K}_3$ consists of the original $K_3$ together with $K_3\times_1 K_2$ and $K_3\times_2 K_2$, which looks like an angular `C' shape. Finally, $K_4$ is the cone over $\widehat{K}_3$, resulting in a filled pentagon. 
		\begin{figure}[H]
			\centering		
			\includegraphics[width=0.85\linewidth]{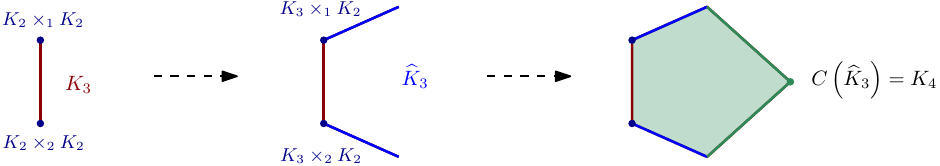}
			\caption{$K_4$ from $K_3$}
		\end{figure}	
        \vspace*{-0.2cm}\item[(iii)] To construct $K_5$ from $K_4$, see that $K_4$ has five boundary edges namely $K_2\times_1 K_3$, $K_3\times_1 K_2$, $K_3\times_3 K_2$, $K_2\times_2 K_3$ and $K_3\times_2 K_2$. Each one of these shares its two boundary endpoints with the other two. See in the below picture that $K_3\times_1 K_2$ and $K_3\times_3 K_2$ have the common boundary point $(K_2\times_2 K_2)\times_1 K_2=(K_2\times_1 K_2)\times_3 K_2$ (by (\ref{eq:2})). Similarly, others are obtained via identification (\ref{eq:1}). Thus in $\widehat{K}_4$, each of the enlarged five cells $K_3\times_1 K_3$, $K_4\times_1 K_2$, $K_4\times_3 K_2$, $K_3\times_2 K_3$ and $K_4\times_2 K_2$ shares two boundary edges (obtained as an enlargement of the common boundary points) with other two. Finally, $K_5$ is the cone over $\widehat{K}_4$. 
		\begin{figure}[H]
			\centering		
			\includegraphics[width=0.95\linewidth]{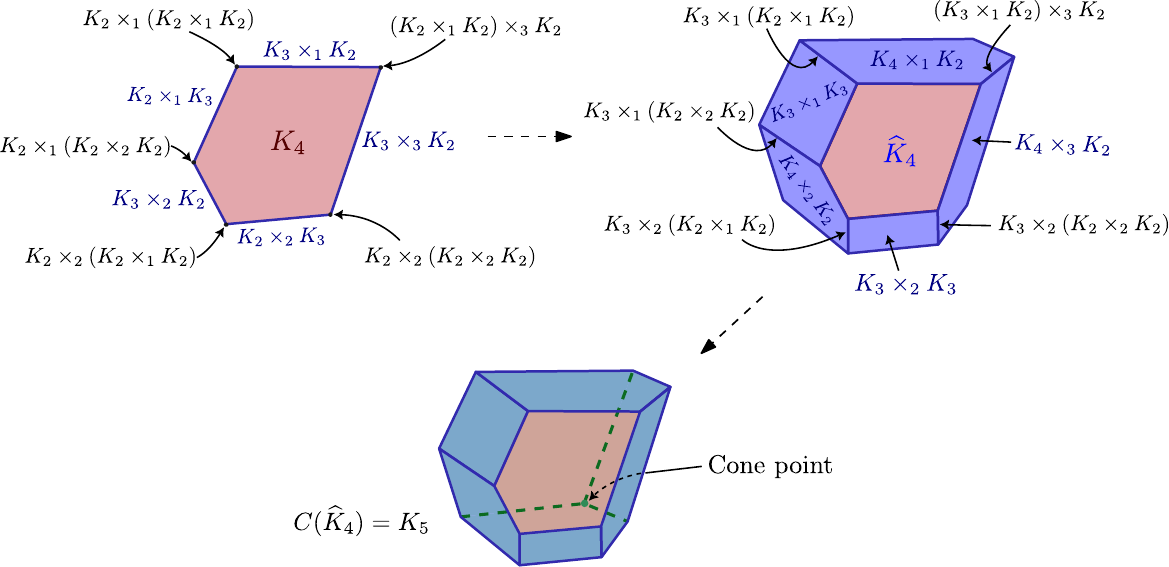}
			\caption{$K_5$ from $K_4$}
		\end{figure}	
	\end{itemize}

	\subsection{Collapsed Multiplihedra}\label{CollapsedMulti}
	Suppose $(X\{m_i\}),(Y,\{m_i^\prime\})$ are two $A_\infty$ spaces and $f: X\to Y$ is a weak homomorphism i.e., there is a homotopy between the maps $f\comp m_2$ and $m_2^\prime\comp (f\times f)$. Such maps are called $H$-maps. In general, there is a notion of $A_n$ maps in Stasheff \cite[II, Def. 4.1]{HAH}, which satisfy $f\comp m_i=m_i^\prime\comp (1\times f^i)$ for $i\leq n$.
	Thus we have a map $f_2:\mathcal{J}(2)\times X^2\to Y$, where $\mathcal{J}(2)$ is an interval. 
	To match things up, rewrite $f$ as $f_1:\mathcal{J}(1) \times X\to Y$, where $\mathcal{J}(1)$ is a single point. 
	Now using $m_2, m_2^\prime, f$, there are six different ways (cf. Figure \ref{fig:JPaintedTrees}) to define a map from $X^3$ to $Y$, namely $f\comp (m_2\comp (m_2\times 1)),$ $f\comp (m_2\comp (1\times m_2)),$ $m_2^\prime\comp (f\times m_2),$ $m_2^\prime\comp (m_2\times f),$ $m_2^\prime\comp (1\times m_2^\prime)\comp(f\times f\times f),$ $m_2^\prime\comp ( m_2^\prime \times 1)\comp(f\times f\times f)$. 
	Using the weak homomorphism of $f$ and weak associativity in $X,Y$ (due to the existence of $m_3$, $m_3^\prime$), one realizes that there are two different homotopies between any two of the six maps. If those two homotopies are homotopic themselves, then we have a map $f_3:\mathcal{J}(3)\times X^3\to Y$, where $\mathcal{J}(3)$ is a filled hexagon. 
	
	If we continue this process, we will get a map $f_n:\mathcal{J}(n)\times X^n\to Y$ for each $n\geq 1$. These complexes $\mathcal{J}(n)$ are called multiplihedra. In the Figure \ref{fig:CollapsedJ_4} below, the dark edges collapse to a point so that the rectangular faces degenerate to edges and the pentagonal face degenerates to a single point, giving rise to Loday's realization of $K_5$. There is a different degeneration from $\mathcal{J}(n)$ to $K_{n+1}$, as shown in \cite[\S 5]{SaneblidzeUmble}; Figure \ref{fig:CollapsedJ_4-1} exhibits this for $\mathcal{J}(4)$. 
	\begin{figure}[H]
		\centering
		\begin{subfigure}[b]{0.3\linewidth}
			\centering	
			\includegraphics[width=0.5 \linewidth]{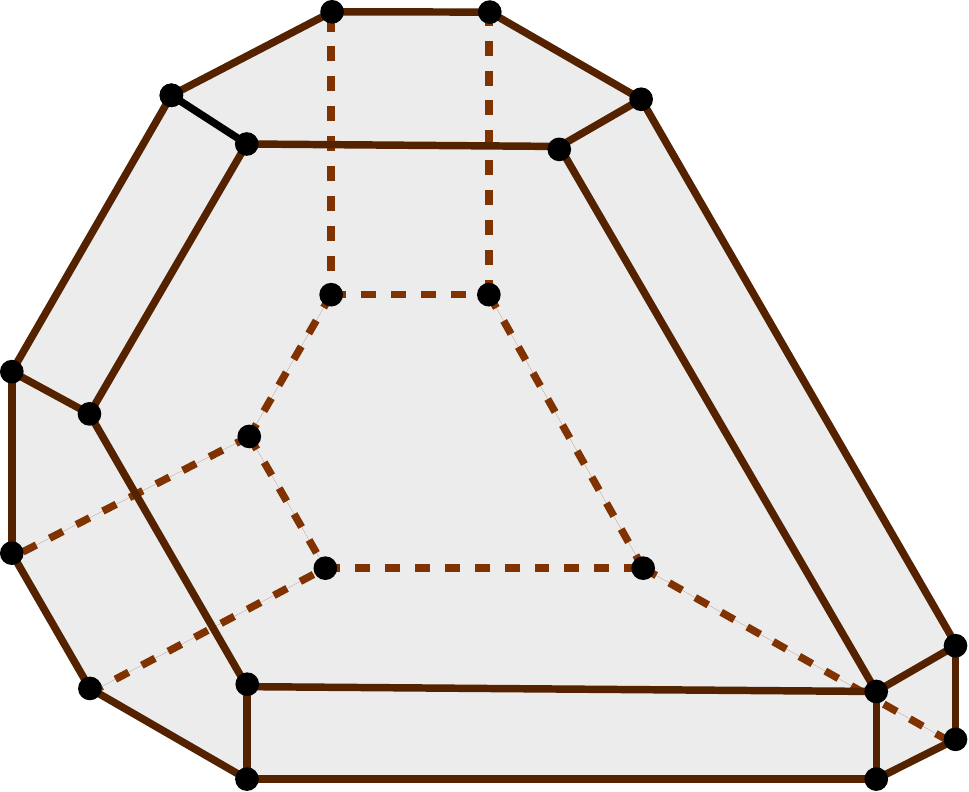}
			\caption{Embedded $\mathcal{J}(4)$ in $\mathbb{R}^3$}
			\label{fig:EmbJ(4)}
		\end{subfigure}
		\hspace{0.05cm}
		\begin{subfigure}[b]{0.3\linewidth}
			\centering	
			\includegraphics[width=0.52 \linewidth]{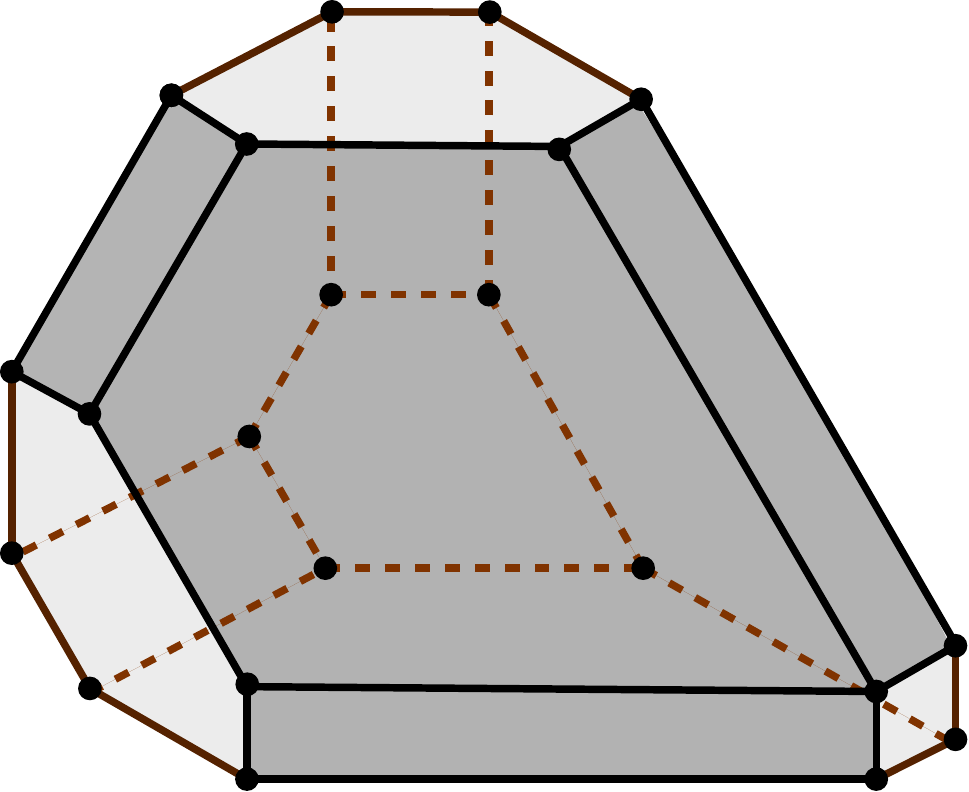}
			\caption{Shaded faces collapsed to get $K_5$}
			\label{fig:CollapsedJ_4}
		\end{subfigure}
		\hspace{0.05cm}
		\begin{subfigure}[b]{0.3\linewidth}
			\centering	
			\includegraphics[width=0.52 \linewidth]{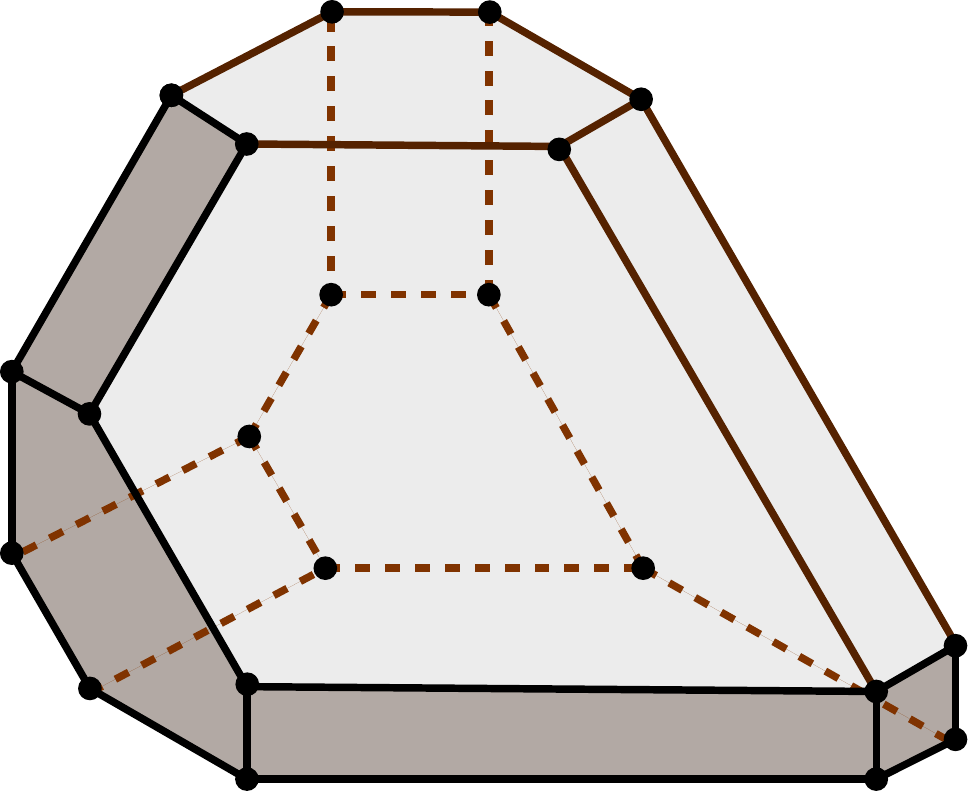}
			\caption{Another degeneration}
			\label{fig:CollapsedJ_4-1}
		\end{subfigure}
		\caption{$\mathcal{J}(4)$ and its degeneration to $K_5$}
	\end{figure}
	
	Multiplihedra first appeared in the work of Stasheff \cite{Stas}. However, in 1986, Norio Iwase and Mamoru Mimura \cite[Section 2]{IwaseMimura} gave the first detailed construction of $\mathcal{J}(n)$ with face operators, and described their combinatorial properties. It was also shown that $\mathcal{J}(n)$ is homeomorphic to the unit cube of dimension $n-1$. Using this description of $\mathcal{J}(n)$, they defined $A_n$ maps. But even before them, Boardman and Vogt \cite{Boardman} (around 1973) had developed several homotopy equivalent versions of a space of \textit{painted binary trees} with interior edges of length in $[0,1]$ to define maps between $A_\infty$ spaces which preserve the multiplicative structure up to homotopy. In 2008, Forcey \cite[Theorem 4.1]{Forcey1} proved that the space of painted trees with $n$ leaves, as convex polytopes, are combinatorially equivalent to the CW-complexes described by Iwase and Mimura. Indeed, Forcey associated a coordinate to each painted binary tree, which generalized the Loday's integer coordinates associated with binary trees corresponding to the vertices of associahedra. Figure \ref{fig:EmbJ(4)} of $\mathcal{J}(4)$ is drawn with such coordinates for the vertices. We shall use the definition of $\mathcal{J}(n)$, as defined in \cite{Forcey1}, in terms of painted trees. 
	\begin{defi}\label{def:PaintedTree}
		A \textit{painted tree} is painted beginning at the root edge (the leaf edges are unpainted), and always painted in such a way that there are only following three types of nodes:
		\begin{figure}[H]
			\centering
			\begin{subfigure}[b]{0.3\linewidth}
				\centering
				\includegraphics[width=0.45\linewidth]{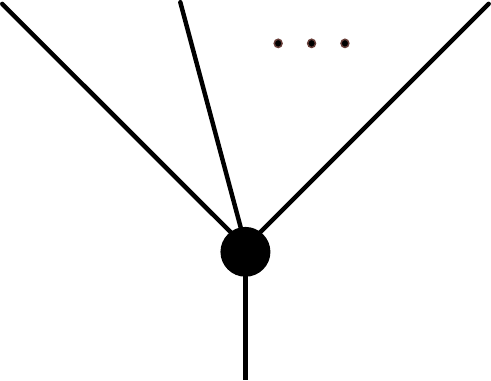}
				\subcaption{} \label{PT2_i}
			\end{subfigure}
			\hspace{0.1cm}
			\begin{subfigure}[b]{0.3\linewidth}
				\centering	\includegraphics[width=0.45\linewidth]{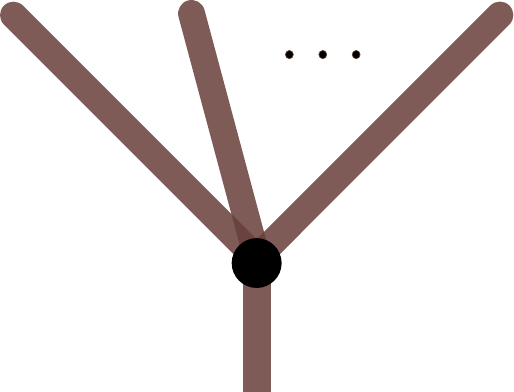}
				\subcaption{} 
				\label{PT2_ii}
			\end{subfigure}
			\hspace{0.1cm}
			\begin{subfigure}[b]{0.3\linewidth}
				\centering	\includegraphics[width=0.45\linewidth]{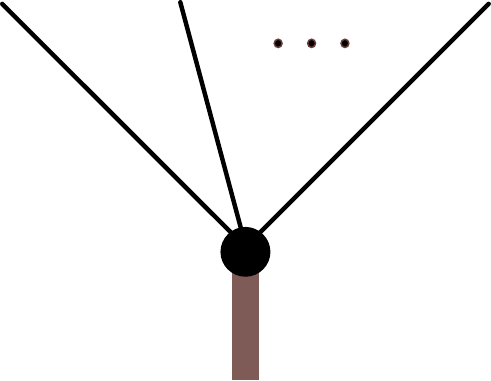}
				\subcaption{} \label{PT2_iii}
			\end{subfigure}
			\caption{Admissible nodes}
		\end{figure}
		\noindent This limitation on nodes implies that painted regions must be connected, and that painting must proceed up every branch of a node.
	\end{defi}
	Let $J(n)$ consist of all painted trees with $n$ leaves. There is a refinement ordering defined as follows.
	\begin{defi}{\cite[Definition 1]{Forcey1}}
		For $t,t^\prime\in J(n)$, we say $t$ \textit{refines} $t'$ and denote by $t\preccurlyeq t^\prime$ if $t^\prime$ obtained from $t$ by collapsing some of its internal edges.\\
		We say $t$ \textit{minimally refines} $t'$ if $t$ refines $t'$ and there is no $s\in J(n)$ such that both $t$ refines $s$ and $s$ refines $t'$.
	\end{defi}
	\noindent Now $(J(n),\preccurlyeq)$ is a poset with painted binary trees as smallest elements (in the sense that nothing refines them) and the painted corolla as the greatest element (in the sense that everything refines it). The $n$-th multiplihedra is defined as follows.
	\begin{defi}
		The $n$-th multiplihedra $\mathcal{J}(n)$ is a convex polytope whose face poset is isomorphic to the poset $(J(n),\preccurlyeq)$ of painted trees with $n$ leaves. 
	\end{defi}
	The explicit inductive construction of these polytopes and the correspondence between the facets of $\mathcal{J}(n)$ and the painted trees follows from \cite[Definition 4]{Forcey1}. For instance, the vertices of $\mathcal{J}(n)$ are in bijection with the painted binary trees with $n$ leaves; the edges are in bijection with those painted trees with $n$ leaves which are obtained by the minimal refinement of painted binary trees with $n$ leaves and they are glued together along the endpoints with matching associated to painted binary trees. In this way, the $(n-2)$-dimensional cells of $\mathcal{J}(n)$ are in bijection with those painted trees which refine to corolla with $n$ leaves. They are glued together along $(n-3)$-dimensional cells with matching to associated painted trees to form the complex $\partial \mathcal{J}(n)$. Finally the $(n-1)$ dimensional complex $\mathcal{J}(n)$ is defined as the cone over $\partial \mathcal{J}(n)$ and it corresponds to the painted corolla with $n$ leaves in the poset $J(n)$.
	\begin{figure}[H]
		\centering
		\includegraphics[width=0.4\linewidth]{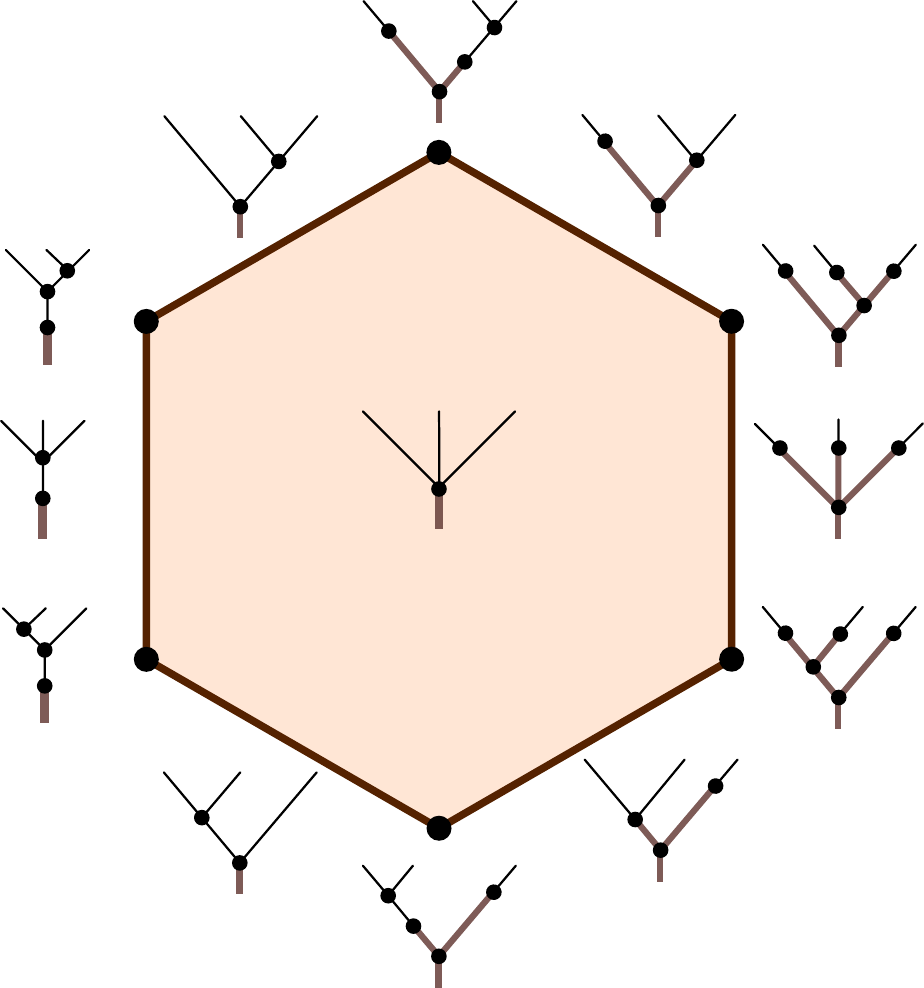}
		\caption{$\mathcal{J}(3)$ labelled by painted trees}
		\label{fig:JPaintedTrees}
	\end{figure}
	\indent We shall give an equivalent description of $\mathcal{J}(n)$ which reflects the promised representation of it stated at the beginning of this subsection. It is given as follows.
	Let $f:A\to B$ be a weak homomorphism (i.e., respects the multiplication in $A$ and $B$ up to homotopy) from an $A_\infty$ space to another $A_\infty$ space. For a given ordered collection $a_1,a_2,...,a_n\in A$, there are three types of elements.
	\begin{enumerate}
		\item[I.] The $f$-image of the elements, which are obtained using different associations of the elements $a_1,a_2,...,a_n$ in $A$. For example, $f(X),$ where $X$ is some rule of association of the elements $a_1,a_2,...,a_n$. 
		\item[II.] The elements obtained using $f$ being homomorphism up to homotopy on the elements of type I and following the same association rule in $B$. For example, if $X=(X_1)((X_2)(X_3))$ is some rule of association of $a_1,a_2,...,a_n$, then elements of the form $f((X_1)\cdot((X_2)(X_3)))$ is of this type. Here $f((X_1)\cdot((X_2)(X_3)))$ denotes the homotopy equivalence between $f((X_1)((X_2)(X_3)))$ and $f(X_1)f((X_2)(X_3))$. Similarly, $f((X_1)\cdot ((X_2)\cdot (X_3)))$, representing the homotopy equivalence between $f((X_1)(X_2\cdot X_3))$ and $f(X_1)f((X_2)\cdot (X_3))$, is also of this type.
		\item[III.] The elements that are obtained using different associations of the elements of type II in $B$. For example, if $X=(X_1)((X_2)((X_3)(X_4)))$ is some rule of association of $a_1,a_2,...,a_n$, then the elements obtained using the different association of $f(X_1), f(X_2), f(X_3), f(X_4)$ in $B$, namely \vspace*{-0.2cm}
		\begin{align*}
			(f(X_1)f(X_2))(f(X_3)f(X_4)),\ f(X_1)f(X_2)&(f(X_3)f(X_4)),\ f(X_1)(f(X_2)f(X_3))f(X_4),\\ (f(X_1)(f(X_2)f(X_3)))f(X_4),&\ 
			f(X_1)f(X_2)f(X_3)f(X_4)
		\end{align*} are of this type. \vspace*{-0.2cm}
	\end{enumerate}
	
	\begin{defi}
		Let $\mathfrak{J}_n$ be the poset of all of the above three types of elements in $B$, ordered such that $P\prec P'$ if $P$ is obtained from $P'$ by at least one of the following operations:
		\begin{enumerate}
			\item \label{it:op1} adding brackets in domain or co-domain elements.
			\item \label{it:op2} replacing `$\cdot$' by $)f($ without changing the association rule in $P'$.
			
			\item \label{it:op3} removing one or more consecutive `$\cdot$' by adding a pair of brackets that encloses all the adjacent elements to all those `$\cdot$' which are removed. In this process, ignore redundant bracketing (if obtained). The requirement of consecutive `$\cdot$' is to ensure allowable bracketing.
		\end{enumerate}
	\end{defi}
	The above operations are to be understood in the following ways:
	\begin{itemize}
		\item For two type I (or III) elements $P,P'$, we say $P\prec P'$ if $P,P'$ follow above operation (\ref{it:op1}) in domain (or co-domain). 
		For example, $f(a_1(a_2(a_3a_4)))\prec f(a_1(a_2a_3a_4)),$ $f(a_1)(f(a_2)f(a_3a_4))\prec f(a_1)f(a_2)f(a_3a_4)$.
		
		\item  For two type II elements $Q,Q'$, we say $Q\prec Q'$ if $Q,Q'$ follow above operation (\ref{it:op2}) or (\ref{it:op3}). 
		For example, $f(a_1)f(a_2\cdot (a_3a_4))\prec f(a_1\cdot a_2\cdot (a_3a_4)),$ $f(a_1\cdot (a_2(a_3a_4)))\prec f(a_1\cdot (a_2\cdot (a_3a_4))).$  
		
		\item For type I element $P$ and type II element $Q$, we say $P\prec Q$ if $P,Q$ follow above operation (\ref{it:op3}).
		
		For example, $f((a_1a_2)(a_3a_4)) \prec f((a_1a_2) \cdot (a_3a_4)),$ $f(a_1a_2a_3a_4) \prec f(a_1\cdot a_2\cdot a_3\cdot a_4)$. 
		
		\item For type II element $Q$ and type III element $P$, we say $P\prec Q$ if $P,Q$ follow above operation (\ref{it:op2}) or (\ref{it:op3}).
		For example, $(f(a_1)f(a_2a_3))f(a_4)\prec f(a_1\cdot (a_2a_3))f(a_4),$ $f(a_1)(f(a_2a_3)f(a_4))\prec f(a_1)(f(a_2\cdot a_3)f(a_4))$.	 
	\end{itemize}
	Now, depending on the poset $(\mathfrak{J}_n,\prec)$, we define another set of complexes $J_n$ for $n\geq 1$.
	\begin{defi}\label{def:Multiplihedra}
		Define $J_n$ to be the convex polytope of dimension $n-1$, whose face poset is isomorphic to $(\mathfrak{J}_n,\prec)$ for $n\geq 1$.
	\end{defi}
	The existence and the equivalence of these complexes with the multiplihedra follows from the following lemma. 
	\begin{lemma}\label{lem:ConRealMulti}
		$J_n$ is isomorphic to the multiplihedron $\mathcal{J}(n)$ for any $n\geq 1$. 
	\end{lemma} 
	\begin{proof}
		It follows from the definitions of $\mathcal{J}(n)$ and $J_n$ that to exhibit an isomorphism between the mentioned complexes, it is enough to provide an isomorphism at the poset level. Define a map $\Phi: J(n)\to \mathfrak{J}_n$ as follows.
		\begin{itemize}
			\item[i)] Put $a_1$ through $a_n$ from left to right above the leaves of a painted tree.
			\item[ii)] If the leaves corresponding to $a_k$ through $a_l$ for $1\leq k< l\leq n$ are joined to a node of type \ref{PT2_i} or of type \ref{PT2_iii}, then associate $(a_ka_{k+1}\ldots a_l)$ (cf. Figure \ref{PT3_i}) or $f(a_k\cdot a_{k+1}\cdot \ldots \cdot a_l)$ (cf. Figure \ref{PT3_iii}) respectively to that node. In case $1\leq k=l\leq n$, then associate $f(a_k)$ to the corresponding node. 
			\item[iii)] Then proceed to the nodes just below the above ones. If a node is of type \ref{PT2_i} or \ref{PT2_iii} joining $X_1$ through $X_m$ as associated nodes just above, then associate $(X_1X_2\ldots X_m)$ or $f(X_1\cdot X_2\cdot \ldots\cdot X_m)$ respectively to that node. If a node is of type \ref{PT2_ii} joining $f(Y_1)$ through $f(Y_m)$ as associated nodes just above, then associate $(f(Y_1)f(Y_2)\ldots f(Y_m))$ to that node (cf. Figure \ref{PT3_ii}).
			\item[iv)] Continue the above step iii) till the root node of a painted tree.
		\end{itemize} 
		\begin{figure}[H]
			\centering
			\begin{subfigure}[b]{0.3\linewidth}
				\centering					
				\includegraphics[width=0.57\linewidth]{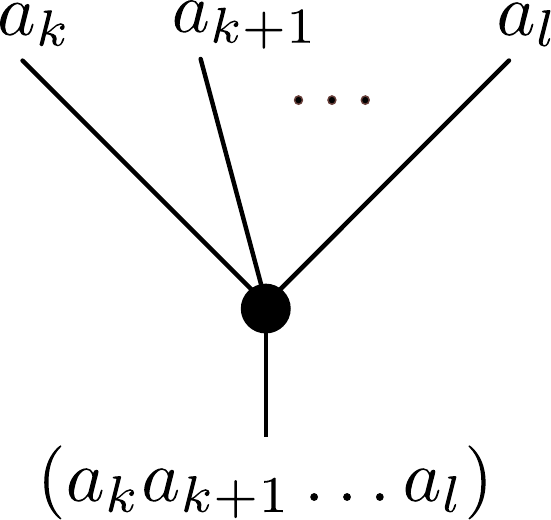}
				\caption{} \label{PT3_i}
			\end{subfigure}
			\hspace{0.1cm}
			\begin{subfigure}[b]{0.3\linewidth}
				\centering	
				\includegraphics[width=0.8\linewidth]{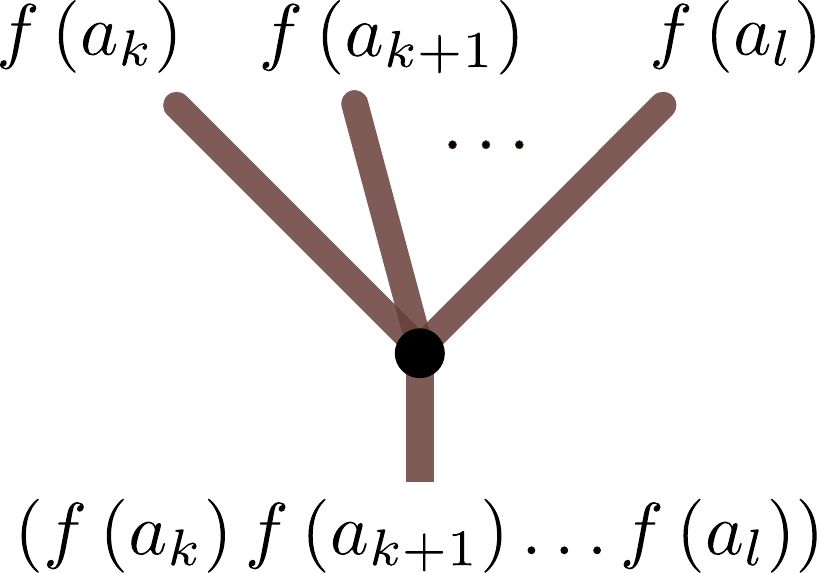}
				\caption{} \label{PT3_ii}
			\end{subfigure}
			\hspace{0.1cm}
			\begin{subfigure}[b]{0.3\linewidth}
				\centering	
				\includegraphics[width=0.64\linewidth]{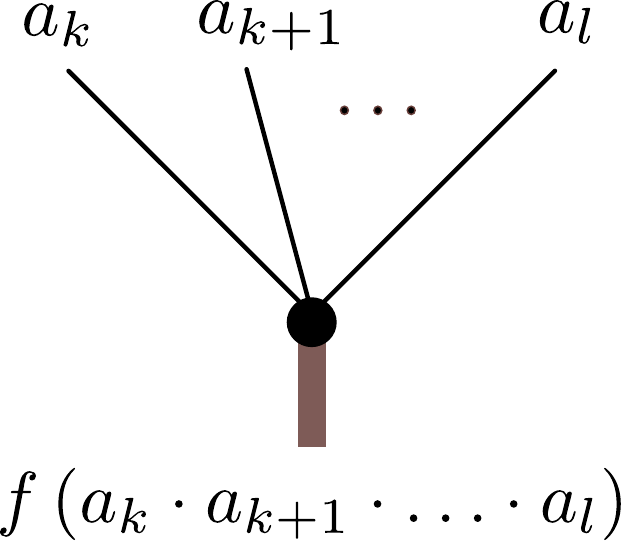}
				\caption{} \label{PT3_iii}
			\end{subfigure}
			\caption{Bijection between the nodes of painted tree and the elements of defined poset}
		\end{figure}
		The element (ignoring redundant brackets, if exist) associated to the root node of a painted tree $t\in J(n)$, is defined to be $\Phi(t)\in \mathfrak{J}_n$. For example, the $\Phi$-image of the painted tree $t\in J(5)$ in Figure \ref{fig:Bijection_eg} is $f(a_1a_2)(f(a_3)f(a_4\cdot a_5))\in \mathfrak{J}_5$.\\
		Note that each painted tree is uniquely determined by its nodes and each position of those nodes associates a unique element. Also, the image of $t\in J(n)$ under $\Phi$ is determined by the associated elements to the nodes of $t$. Thus $\Phi$ maps each element of $J(n)$ to a unique element of $\mathfrak{J}_n$ and hence $\Phi$ is a bijection.
		\begin{figure}[H]
			\centering
			\includegraphics[width=0.35\linewidth]{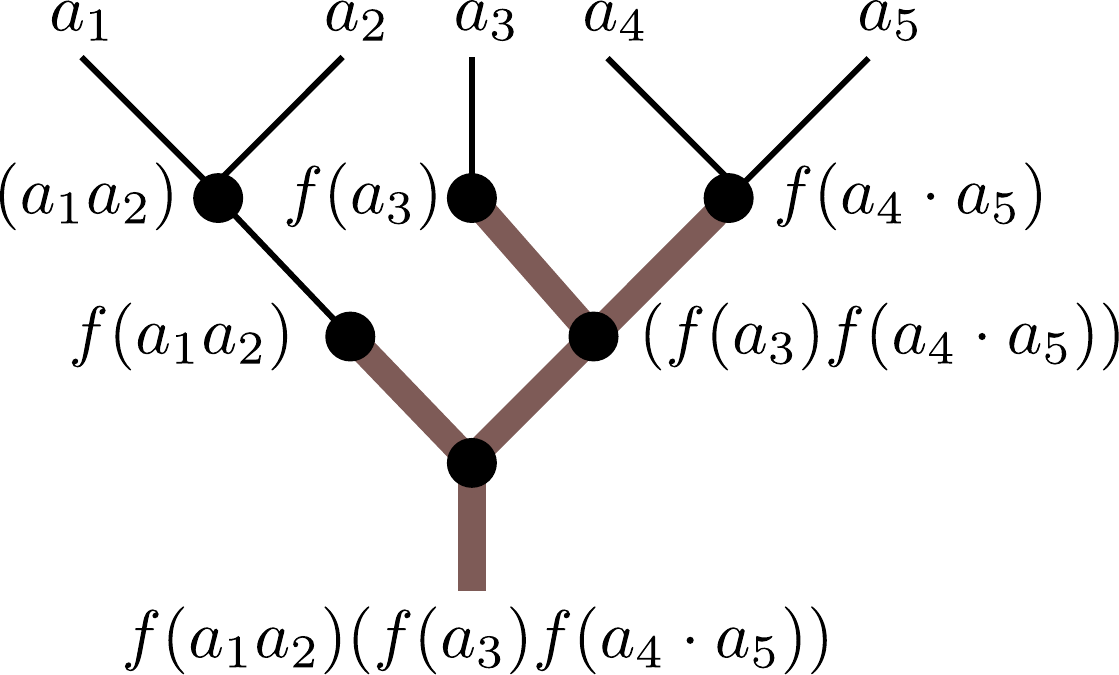}
			\caption{Elements associated to the nodes}
			\label{fig:Bijection_eg}
		\end{figure}
		It remains to check that $\Phi$ preserves the partial order. By the definition of $\preccurlyeq$, it is enough to show that $\Phi(t)\prec \Phi(t')$ when $t\preccurlyeq t'$ minimally. If $t\preccurlyeq t'$ minimally, then $t'$ is obtained from $t$ by collapsing an unpainted internal edge or a painted internal edge or a bunch of painted edges. Note that collapsing an unpainted internal edge results in either the removal of brackets in the domain (operation (\ref{it:op1}) in $\mathfrak{J}_n$) or the addition of one or more $\cdot$ by removing brackets (operation (\ref{it:op3}) in $\mathfrak{J}_n$). Collapsing a painted internal edge results in the removal of brackets in the co-domain (operation (\ref{it:op1}) in $\mathfrak{J}_n$) while collapsing a bunch of painted edges results in replacing $)f($ by $\cdot$ (operation (\ref{it:op2}) in $\mathfrak{J}_n$). In all the cases $\Phi(t)\prec \Phi(t')$, completing the proof.
	\end{proof}
	\noindent Using this lemma, we consider $J_n$ (Definition \ref{def:Multiplihedra}) as the $n$-th multiplihedron. The pictures of $J_1,$ $J_2,$ $J_3$ are depicted later in Figure \ref{fig:multiplihedra}, with labelling of the faces in terms of elements of $\mathfrak{J}(1)$, $\mathfrak{J}(2)$, $\mathfrak{J}(3)$ respectively.  

	\begin{figure}[h]
		\centering
		\begin{subfigure}[b]{0.08\linewidth}
			\centering
			\includegraphics[width=1.25cm]{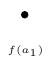}
			\caption{$J_1$}
		\end{subfigure}
		\hspace{0.5cm}
		\begin{subfigure}[b]{0.13\linewidth}
			\includegraphics[width=\linewidth]{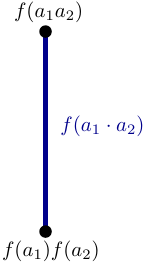}
			\caption{$J_2$}
		\end{subfigure}
		\hspace{1cm}
		\begin{subfigure}[b]{0.5\linewidth}
			\includegraphics[width=\linewidth]{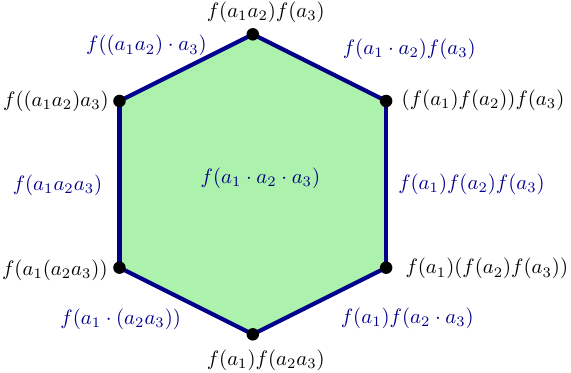}
			\caption{$J_3$}
		\end{subfigure}
		\caption{Multiplihedra}
		\label{fig:multiplihedra}
	\end{figure}
 
	Now suppose $B$ is an associative space. Due to the associativity in $B$, there will be only one element of type III (as defined before) for each association rule of $a_1,a_2,...,a_n.$ For example, if $X=((X_1X_2)(X_3X_4))$ is some association rule of $a_1,a_2,...,a_n$, then there is only one element $f(X_1)f(X_2)$ $f(X_3)f(X_4)$ in $B$ using the fact that $f$ is a homomorphism up to homotopy. We will call them degenerate type III elements. 
	\begin{defi}\label{def:CollMulti}
		Let $\mathfrak{J}'_n$ be the poset of all type I, type II, and degenerate type III elements in $B$ with the ordering induced from $(\mathfrak{J}'_n,\prec)$. 
		We define the \textit{collapsed multiplihedron} $J'_n$ to be a cellular complex of dimension $n-1$, whose face poset is isomorphic to $\mathfrak{J}'_n.$ 
	\end{defi} 
	\noindent As the posets $\mathfrak{J}'_n$ are obtained by the degeneracy of certain elements in $\mathfrak{J}_n$, the complexes $J'_n$ are obtained by collapsing certain faces of $J_n$. Thus the existence of the complexes $J'_n$ is guaranteed by the existence of multiplihedron $J_n$. We will use this definition to show that $J'_n$ is combinatorially isomorphic to the associahedron $K_{n+1}$ in \S \ref{StasMulti}.

	\subsection{Graph Cubeahedra and Design Tubings}\label{Cubeahedra}
	Devadoss \cite{Devadoss1} gave an alternate definition of $K_{n}$ with respect to tubings on a path graph.
	\begin{defi}[Tube]
		Let $\Gamma$ be a graph. A \textit{tube} is a proper nonempty set of nodes of $\Gamma$ whose induced graph is a proper, connected subgraph of $\Gamma$. 
	\end{defi}	
	There are three ways that two tubes $t_1$ and $t_2$ may interact on the graph. 
	\begin{itemize}
		\item $t_1$ and $t_2$ are \textbf{nested} if $t_1\subset t_2$ or $t_2\subset t_1$. 
		\begin{figure}[H]
			\centering
			\includegraphics[scale=0.7]{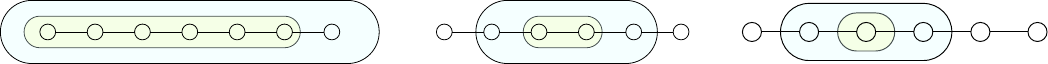}
			\caption{Nested tubes}
		\end{figure}
		\item $t_1$ and $t_2$ \textbf{intersect} if $t_1\cap t_2\neq \phi$ and $t_1\nsubseteq t_2$ and $t_2\nsubseteq t_1$. 
		\begin{figure}[H]
			\centering
			\begin{subfigure}[b]{0.375\linewidth}
				\centering
				\includegraphics[width=0.6\linewidth]{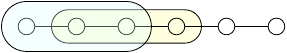}
			\end{subfigure}
			\hspace{0.5cm}
			\begin{subfigure}[b]{0.32\linewidth}
				\includegraphics[width=0.6\linewidth]{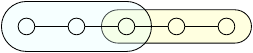}
			\end{subfigure}
			\caption{Intersection of tubes}
		\end{figure}
		\item $t_1$ and $t_2$ are \textbf{adjacent} if $t_1\cap t_2= \phi$ and $t_1\cup t_2$ is a tube. 
		\begin{figure}[H]
			\centering
			\begin{subfigure}[b]{0.325\linewidth}
				\centering
				\includegraphics[width=0.8\linewidth]{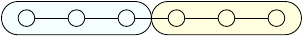}
			\end{subfigure}
			\hspace{0.5cm}
			\begin{subfigure}[b]{0.32\linewidth}
				\includegraphics[width=0.8\linewidth]{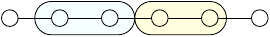}
			\end{subfigure}
			\caption{Adjacent tubes}
		\end{figure}
	\end{itemize}	
	Two tubes are \textbf{compatible} if they are neither adjacent nor intersect i.e., $t_1$ and $t_2$ are compatible if they are nested or $t_1\cap t_2= \phi$ with $t_1\cup t_2$ are not tubes.
	\begin{defi}
		A \textit{tubing} $T$ of $\Gamma$ is a set of tubes of $\Gamma$ such that every pair of tubes in $T$ is compatible. A \textit{$k$-tubing} is a tubing with $k$ tubes. 
	\end{defi}
	\noindent A few examples of tubings are given below.
	\begin{figure}[H]
		\centering
		\begin{subfigure}[b]{0.29\linewidth}
			\centering
			\includegraphics[width=\linewidth]{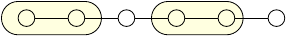}
			\caption*{2-tubing}
		\end{subfigure}
		\hspace{0.5cm}
		\begin{subfigure}[b]{0.29\linewidth}
			\includegraphics[width=\linewidth]{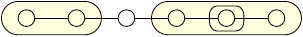}
			\caption*{3-tubing}
		\end{subfigure}
		\hspace{0.5cm}
		\begin{subfigure}[b]{0.29\linewidth}
			\includegraphics[width=\linewidth]{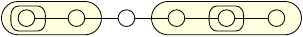}
			\caption*{4-tubing}
		\end{subfigure}
		\caption{Tubings}
	\end{figure}
	
	If we think of the $n-1$ nodes of a path graph $\Gamma$ as dividers between the $n$ letters of a word and the tube as a pair of parentheses enclosing the letters, then the compatibility condition of the tubes corresponds to the permissible bracketing of the word. Now using the combinatorial description (cf. Definition \ref{def:Associahedra1}) of $K_n$, one has the following result.
	\begin{lemma}{\cite[Lemma 2.3]{Devadoss2}}
		\,Let $\Gamma$ be a path graph with $n-1$ nodes. The face poset of $K_n$ is isomorphic to the poset of all valid tubings of $\Gamma$, ordered such that tubings $T\prec T'$ if $T$ is obtained from $T'$ by adding tubes.
	\end{lemma}
	\noindent On a graph, Devadoss \cite{Devadoss1} defines another set of tubes called design tubes. 
	\begin{defi}[Design Tube]
		Let $G$ be a connected graph.
		A \textit{round tube} is a set of nodes of $G$ whose induced graph is a connected (and not necessarily proper) subgraph of $G$.
		A \textit{square tube} is a single node of $G$. Then round tubes and square tubes together called \textit{design tubes} of $G$.
	\end{defi}
	
	\noindent Two design tubes are compatible if
	\begin{enumerate}
		\item they are both round, they are not adjacent and do not intersect;
		\item otherwise, they are not nested.
	\end{enumerate}
	
	\begin{defi}[Design Tubing]
		A \textit{design tubing} $U$ of $G$ is a collection of design tubes of $G$ such that every pair of tubes in $U$ is compatible.
	\end{defi}
	
	\begin{figure}[H]
		\centering
		\begin{subfigure}[b]{0.29\linewidth}
			\centering
			\includegraphics[width=0.8\linewidth]{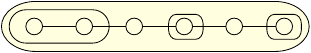}
			\caption*{4-design tubing}
		\end{subfigure}
		\hspace{0.5cm}
		\begin{subfigure}[b]{0.29\linewidth}
			\includegraphics[width=\linewidth]{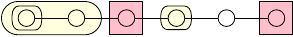}
			\caption*{5-design tubing}
		\end{subfigure}
		\hspace{0.5cm}
		\begin{subfigure}[b]{0.29\linewidth}
			\includegraphics[width=\linewidth]{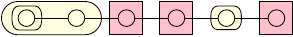}
			\caption*{6-design tubing}
		\end{subfigure}
		\caption{Design tubings}
	\end{figure}
	\noindent
	Note that, unlike ordinary tubes, round tubes do not have to be proper subgraphs of $G$. \\
	\hspace*{0.5cm}Based on design tubings, Devadoss \cite{Devadoss1} constructed a set of polytopes called graph cubeahedra. For a graph $G$ with $n$ nodes, define $\boxdot_G$ to be the \textit{$n$-cube} where each pair of opposite facets corresponds to a particular node of $G$. Specifically, one facet in the pair represents that node as a round tube and the other represents it as a square tube.
	Each subset of nodes of $G$, chosen to be either round or square, corresponds to a unique face of $\boxdot_G$ defined by the intersection of the faces associated with those nodes. The empty set corresponds to the face which is the entire polytope $\boxdot_G$.
	
	\begin{defi}[Graph Cubeahedron]
		For a graph $G$, truncate faces of $\boxdot_G$ which correspond to round tubes in increasing order of dimension. The resulting polytope $\mathcal{C}G$ is the \textit{graph cubeahedron}.
	\end{defi}
	\noindent The graph cubeahedron $\mathcal{C}G$ can also be described as a convex polytope whose face poset formed through the design tubings. 
	\begin{thm}{\cite[Theorem 12]{Devadoss1}}\label{thm:OrderCubeahedra}
		For a graph $G$ with $n$ nodes, the graph cubeahedron $\mathcal{C} G$ is a simple convex polytope of dimension $n$ whose face poset is isomorphic to the set of design tubings of $G$, ordered such that $U \prec U^{\prime}$ if $U$ is obtained from $U^{\prime}$ by adding tubes. 
	\end{thm}
	\noindent In this article, we are interested in the case when $G$ is a path graph. We will make use of the above theorem to show a combinatorial isomorphism between $\mathcal{C}G$ for $G$ is a path graph with $n$ nodes and multiplihedra $J_{n+1}$ in \S \ref{MultiCubea}.
	
	\section{Isomorphisms Between The Four Models}\label{Main}
	
	We prove the main result of this paper in this section. 
	\begin{thm}\label{thm:Main}
		The four models of associahedra: Stasheff complexes, cellular complexes obtained by Loday's cone construction, collapsed multiplihedra, and graph cubeahedra for path graphs are all combinatorially isomorphic.
	\end{thm}
	\begin{proof}
		We prove the isomorphisms in the next three subsections. In \S \ref{LodayStasheff} we prove that the complexes obtained via the cone construction of Loday are combinatorially isomorphic to the Stasheff complexes (Theorem \ref{thm:StasheffLoday}). In \S \ref{StasMulti} we prove that the Stasheff complexes and collapsed multiplihedra are isomorphic (Proposition \ref{thm:StasMulti}). Finally, in \S \ref{MultiCubea}, the isomorphism between the collapsed multiplihedra and graph cubeahedra is shown (Proposition \ref{thm:MultiCubea}). Combining all three, we have our required result.
	\end{proof}
	
	\subsection{Loday's construction vs Stasheff complexes}\label{LodayStasheff}
	By Stasheff's description, $K_{n+1}$ is the cone over its boundary elements $K_p\times_r K_q$ for $p+q=n+2$, $2\leq p\leq n$ and $r=1,2,\ldots, p.$ On the other hand, consider $C(\widehat{K}_n)$, where $\widehat{K}_n$ consists of the initial $K_n$ together with $K_{p+1}\times_r K_q$ such that $p+q=n+1$, $2\leq p\leq n-1$ and $r=1,2,\ldots, p$. This enlargement $\widehat{K}_n$ can be described in terms of bracketing as follows.
	\begin{itemize}
		\item $K_n$ corresponds to $0$-bracketing of the word $x_1x_2\ldots x_n$ i.e., the word itself or the trivial bracketing $(x_1x_2\ldots x_n)$. The immediate faces i.e., the boundary consists of $K_p\times_r K_q$ with $p+q=n+1$, $2\leq p\leq n-1$ and $r=1,2,\ldots,p$. Now $K_p\times_r K_q$ corresponds to the $1$-bracketing $x_1\ldots x_{r-1} (x_r\ldots x_{r+q-1})x_{r+q}\ldots x_n$.
		\item The enlargement $\widehat{K}_n$ corresponds to the adding of a letter $x_{n+1}$ to the right of the bracketing corresponding to $K_n$. Then the bracketing $x_1\ldots x_{r-1} (x_r\ldots x_{r+q-1})x_{r+q}\ldots x_n$ extends to $x_1\ldots x_{r-1} (x_r\ldots x_{r+q-1})x_{r+q}\ldots x_nx_{n+1}$ for each $p,q,r$ such that $p+q=n+1$, $2\leq p\leq n-1$, and $r=1,2,\ldots, p$. Also the initial $K_n$ i.e., $(x_1x_2\ldots x_n)$ extends to $(x_1x_2\ldots x_n)x_{n+1}$, which corresponds to $K_2\times_1 K_n$ in $K_{n+1}$. 
		\item Finally one takes cone over the enlarged complex to obtain $K_{n+1}.$
	\end{itemize}
	From the above description, $\widehat{K}_n$ can be thought of as union of $K_p\times_r K_q$ with $p+q=(n+1)+1$ for $2\leq p\leq n$ and $r=1,2,\ldots, p-1$. Thus $\widehat{K}_n$ is a part of the boundary of $K_{n+1}$ (following Stasheff's description).
	
	\begin{thm}\label{thm:StasheffLoday}
		Stasheff complexes are combinatorially isomorphic to Loday's cone construction of associahedra.
	\end{thm}
	
	To prove combinatorial isomorphism between the two mentioned models, we must show bijective correspondence between vertices, edges, and faces of each codimension for both models respecting the adjacencies. But the faces of codimension more than $1$ are contained in the faces of codimension $1$. Thus if we have an appropriate bijection between the faces of codimension $1$ respecting the adjacencies for both models, then the resulting models being cone over combinatorially isomorphic codimension $1$ faces, they are combinatorially isomorphic.
	
	\begin{proof}[Proof of Theorem \ref{thm:StasheffLoday}]
		
		It is enough to show that the boundary of $K_{n+1}$ in Loday's construction can be subdivided to match them with the boundary elements $K_p\times_r K_q$ of $K_{n+1}$ in Stasheff model for $p+q=n+2$, $2\leq p\leq n$ and $r=1,\ldots, p$. 
		As observed in the initial discussion, the only missing boundary part of $K_{n+1}$ in Loday's construction is the union of $K_p\times_p K_q$ for $p+q=n+2$ with $2\leq p\leq n$. Note that all these missing faces are adjacent to a common vertex, which corresponds to the right to left $(n-1)$-bracketing $x_1(x_2(\ldots (x_{n-1}(x_n x_{n+1}))...))$. As there are $\binom{n-1}{n-2}=n-1$ many choices for removing $(n-2)$ brackets from a $(n-1)$-bracketing (that corresponds to the vertices of $K_{n+1}$), each vertex of $K_{n+1}$ is adjacent to exactly $n-1$ faces of codimension $1$ of $K_{n+1}$ (by poset description of Stasheff's $K_{n+1}$). So the vertex corresponding to $x_1(x_2(\ldots (x_{n-1}(x_n x_{n+1}))...))$ is not obtained in $\widehat{K}_n$. Now if we consider any other $(n-1)$-bracketing, then there can be at most $n-2$ parentheses after $x_{n+1}$. So removing those parentheses along with some others, we can get a $1$-bracketing that does not enclose $x_{n+1}$ i.e., those vertices are adjacent to some $K_p\times_r K_q$ for $p+q=n+2$ and $r=1,2,\ldots,p-1$. Thus any vertex of $K_{n+1}$ except that corresponding to $x_1(x_2(\ldots (x_{n-1}(x_n x_{n+1}))...))$  is present in $\widehat{K}_n$. We identify this missing vertex with the coning vertex of Loday's construction.\\
		\hspace*{0.5cm}We shall prove that the missing faces of $K_{n+1}$ in $C(\widehat{K}_n)$ can be realized as a cone over some portion of the boundary of $\widehat{K}_n$. Then we will divide the part $C(\partial \widehat{K}_n)$ accordingly to identify those with the missing faces. We will prove this together with the final result by induction on the following statements: 
		
		\begin{itemize}
			\item[I.] $Q_{n-3}:$ $K_p \times_r K_q = C\left((\widehat{K}_{p-1}\times_r K_q) \cup (K_p\times_r \widehat{K}_{q-1})\right) \text{ if } p+q=n+1 \text{ and }p,q\geq 3$.
			
			\item[II.] $P_{n-2}:$ $K_{n}=C\left( \widehat{K}_{n-1}\right)$, $n \geq 3.$
		\end{itemize}
		Here the equalities in the statements represent a combinatorial isomorphism. Note that $Q_{n-3}$ is a collection of statements and the index $r$ is superfluous. We will use the convention that $\widehat{K}_1=\varnothing$, $C(\varnothing)=\{\ast\}$ and allow $p,q\geq 2$. Then $Q_{n-3}$ contains the statement for $K_{n-1}\times_r K_2$ as well as $K_2\times_r K_{n-1}$. Moreover, these are equivalent to the statement $P_{n-3}$ since $K_2$ is a point and $K_{n-1}\times K_2$ is $K_{n-1}$.\\
		\hspace*{0.5cm}The steps of induction are as follows.\\
		
		\noindent\textit{Step 0:} \textit{Show that $P_1$ holds.}\\
		Note that $K_2$ is a point that parametrizes the binary operation. As a point has no boundary, so $\widehat{K}_2$ is also a point, and $C(\widehat{K}_2)$ is an interval. Now $K_3$ parametrizes the family of $\mathrm{3}$-ary operations that relate the two ways of forming a $\mathrm{3}$-ary operation via a given binary operation. Thus, $K_3$ also represents an interval. Here the boundary of $K_3$ consists of two points $K_2\times_1 K_2$ and $K_2\times_2 K_2$. Let us map $K_2\times_1 K_2$ and $K_2\times_2 K_2$ to $\widehat{K}_2$ and the coning point in $C(\widehat{K}_2)$ respectively. Then we can map the other points of $K_3$ linearly to $C(\widehat{K}_2)$. Thus we get $K_3$ and $C(\widehat{K}_2)$ are combinatorially isomorphic. So $P_1$ is true.\\

		\noindent \textit{Step 1:} \textit{Assuming that $P_1$ through $P_{n-4}$ hold, show that $Q_{n-3}$ holds.}\\
		To prove it we will use the following lemma, the proof of which is given at the end of this subsection.
		\begin{lemma}\label{lemma:CrossCone}
			There is a natural homeomorphism
			\begin{displaymath}
				C(X)\times C(Y)\equiv C\left((X\times C(Y)) \cup (C(X)\times Y)\right),
			\end{displaymath}
			where $x_0,y_0$ are cone points for $C(X),C(Y)$ respectively and $(x_0,y_0)$ is the cone point for $C(Z)$, where $Z=(C(X)\times Y)\cup (X\times C(Y))$.
		\end{lemma}
		Now assuming $P_1$ through $P_{n-4}$, we have $K_l=C(\widehat{K}_{l-1})$ for $l=3,4,...,n-2$. Take any $p,q\geq 3$ with $p+q=n+1$ i.e., $p,q$ both ranges through $3$ to $n-2$. So
		\begin{align*}
			K_p\times_r K_q = &\ C(\widehat{K}_{p-1})\times_r C(\widehat{K}_{q-1})\ \ (\text{by the assumption})\\  
			= &\ C\left((\widehat{K}_{p-1}\times_r C(\widehat{K}_{q-1}))\cup (C(\widehat{K}_{p-1})\times_r \widehat{K}_{q-1})\right)\ \ (\text{by the Lemma \ref{lemma:CrossCone}})\\
			= &\ C((\widehat{K}_{p-1}\times_r K_q)\cup (K_p\times_r \widehat{K}_{q-1}))\ \ (\text{by the assumption})
		\end{align*}
		This shows that $Q_{n-3}$ is true.\\
		
		\noindent \textit{Step 2:} \textit{Assuming $P_1$ through $P_{n-3}$, show that $P_{n-2}$ hold.}\\
		 
		As discussed earlier, to prove that $P_{n-2}$ is true, it is enough to show $K_s\times_s K_t$ with $s+t=n+1$ for $s,t\geq 2$ can be obtained from $C(\widehat{K}_{n-1})$. Consider $s,t\geq 2$ with $s+t=n+1$. Then using the conventions $\widehat{K}_1=\varnothing$ and $C(\varnothing)=\{*\}$, we can write
		\begin{align*}
			&\ K_s\times_s K_t\\
			= &\ C(\widehat{K}_{s-1})\times_s C(\widehat{K}_{t-1})\ (\text{by $P_1$ through $P_{n-3}$})\\
			= &\ C\left((\widehat{K}_{s-1}\times_s K_t)\cup (K_s\times_s \widehat{K}_{t-1})\right)\ (\text{by the Lemma \ref{lemma:CrossCone}})\\
			= &\ C\left( \left\{\bigcup_{(p,q,r)\in V_s}\left((K_p\times_r K_q)\times_s K_t\right)\right\}\bigcup \left\{\bigcup_{(p,q,r)\in V_t}\left(K_s\times_s (K_p\times_r K_q)\right)\right\} \right)			
			(\text{by definition of } \widehat{K}_{i-1}),\\[4pt]
			&\text{ where }V_i=\{(a,b,c)\in \mathbb{N}^3: 2\leq a\leq i-1,\ a+b=i+1,\  1\leq c\leq a-1\},\ i=s,t.
		\end{align*}
	
		Now using equation (\ref{eq:2}) (in \S \ref{Stasheff}), we can write 
		$$(K_p\times_r K_q)\times_s K_t=(K_p\times_{s-q+1} K_t)\times_r K_q$$
		(obtained by substituting $r=p,s=q,t=t,k=r,j=s-q+1$) for the terms in the first set of unions. As $K_p\times_{s-q+1} K_t$ is a face of $K_{p+t-1}$, so $(K_p\times_{s-q+1} K_t)\times_r K_q$ is a face of $K_{p+t-1} \times_r K_q$, which is again a face of $K_{n}$ because for $(p,q,r)\in V_s$, $$(p+t-1)+q=p+q+t-1=s+1+t-1=s+t=n+1.$$
		
		Thus $(K_p\times_r K_q)\times_s K_t$ is a face of $K_{p+t-1} \times_r K_q$ of codimension $1$. But as $t\geq 2$ and $1\leq r\leq p-1$, so $r<p+t-1$, which implies that the face $K_{p+t-1} \times_r K_q$ is already present in the enlargement $\widehat{K}_{n-1}$. Thus each term in the first set of unions is already present in $\widehat{K}_{n-1}$.\\
		
		Similarly, using equation (\ref{eq:1}), we have the identification 
		$$K_s\times_s (K_p\times_r K_q) = (K_s\times_{s} K_p)\times_{s+r-1} K_q$$
		(obtained by substituting $r=s,s=p,t=q,k=r,j=s$) for the terms in the second set of unions. Here $(K_s\times_{s} K_p)\times_{s+r-1} K_q$ is a face of $K_{s+p-1}\times_{s+r-1} K_q$, which is a face of $K_n$ because for $(p,q,r)\in V_t$,
		$$(s+p-1)+q=s-1+(p+q)=s-1+t+1=s+t=n+1.$$
		
		Thus $(K_s\times_s K_{p}) \times_{s+r-1} K_q$ is a face of $K_{s+p-1} \times_{s+r-1} K_q$ of codimension $1$. But $r\leq p-1<p$ implies $s+r-1<s+p-1$, which further implies that the face $K_{s+p-1} \times_{s+r-1} K_q$ is already present in the enlargement $\widehat{K}_{n-1}$. Thus each term in the second set of unions is also present in $\widehat{K}_{n-1}$.\\
	
		It follows that all the parts in the unions are present as a part of the boundary of $\widehat{K}_{n-1}$. Thus the cone over that particular part of the boundary of $\widehat{K}_{n-1}$, we will get $K_s\times_s K_t$ for all $s,t\geq 2$ (with $s+t=n+1$). Also, these are present as a part of boundary of $C(\widehat{K}_{n-1})$. Therefore we get a bijection between the faces (of codimension $1$) of $K_n$ and $\widehat{K}_{n-1}$. Consequently, they are combinatorially isomorphic. So $P_{n-2}$ is true. This completes the induction step as well as the proof of the theorem.
	\end{proof}
	\begin{remark}
		In the above isomorphism, we mapped the starting $K_n$ to $K_2\times_1 K_n$ and the extension of the boundary element $K_p\times_r K_q$ to $K_{p+1}\times_r K_q$. Similarly we could map the starting $K_n$ to $K_2\times_2 K_n$ and the extension of the boundary $K_p\times_r K_q$ to $K_{p+1}\times_{r+1} K_q$. But if we want to map the starting $K_n$ to $K_n\times_r K_2$ ($r=1,2,...,n$), the corresponding extension of boundary $K_p\times_t K_q$ should map to
		$$\begin{cases}
			K_p\times_t K_{q+1} & \text{if }t\leq r\leq t+q-1\\
			K_{p+1}\times_t K_q & \text{if }r>t+q-1\\
			K_{p+1}\times_{t+1} K_q & \text{if }r<t.
		\end{cases}$$
		With a slight modification in the above proof, one can similarly prove that this produces an isomorphism. This, in turn, implies that the faces $K_n\times_r K_2$ or $K_2\times_r K_n$ of $K_{n+1}$ are all equivalent from the point of view of Loday's construction.
	\end{remark}
	We end this subsection with the proof of Lemma \ref{lemma:CrossCone}.
	\begin{proof}[Proof of Lemma \ref{lemma:CrossCone}]
		We will prove the equality by showing both inclusions. First suppose $(x,y)=t(x_0,y_0)+(1-t)(x_1,y_1)\in C(Z)$, where $t\in [0,1]$ and $(x_1,y_1)\in Z$.
		Without loss of generality suppose $(x_1,y_1)\in C(X)\times Y$ i.e., $x_1=t'x_0+(1-t')x'_1$ for some $t'\in [0,1]$ and $x'_1\in X$.
		So 
		\begin{align*}
			(x,y)
			=&\ (tx_0+(1-t)x_1,ty_0+(1-t)y_1)\\
			=&\ (tx_0+(1-t)t'x_0+(1-t)(1-t')x'_1,ty_0+(1-t)y_1)\\
			=&\ ((1-(1-t)(1-t'))x_0+(1-t)(1-t')x'_1,ty_0+(1-t)y_1)\\
			=&\ (t_1x_0+(1-t_1)x'_1,ty_0+(1-t)y_1)
			\in  C(X)\times C(Y)
		\end{align*}
		and $t_1=1-(1-t)(1-t')$. This implies that $C(Z)\subseteq C(X)\times C(Y).$
		\begin{figure}[H]
			\centering
			\includegraphics[width=0.55\linewidth]{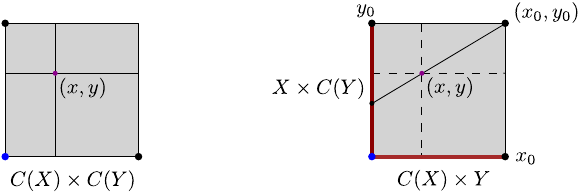}
			\caption{Visual proof when $X=Y=$ point}
		\end{figure}
		Conversely let $(x,y)=(t_1x_0+(1-t_1)x_1,t_2y_0+(1-t_2)y_1)\in C(X)\times C(Y)$ for some $x_1\in X$, $y_1\in Y$ and $t_1,t_2\in [0,1]$. Now consider the following cases\\
		\noindent
		\textit{Case I}: $t_1=t_2=t$.
		$$(x,y)=t(x_0,y_0)+(1-t)(x_1,y_1)\in C(Z).$$
		\textit{Case II}: $t_1>t_2$.
		\begin{align*}
			(x,y)=&\  t_2(x_0,y_0)+(1-t_2)\left( \textstyle{\frac{t_1-t_2}{1-t_2}}x_0+\textstyle{\frac{1-t_1}{1-t_2}}x_1,y_1\right)\\
			=&\  t_2(x_0,y_0) +(1-t_2) (t'x_0+ (1-t') x_1,y_1)\in  C(Z),\\ 
			&\ \text{ where }t'= \textstyle{\frac{t_1-t_2}{1-t_2}}.
		\end{align*}
		\textit{Case III}: $t_1<t_2$.
		\begin{align*}
			(x,y)=&\  t_1(x_0,y_0)+(1-t_1)\left(x_1, \textstyle{\frac{t_2-t_1}{1-t_1}}y_0+ \textstyle{\frac{1-t_2}{1-t_1}}y_1\right)\\
			=&\  t_1(x_0,y_0)+(1-t_1)(x_1,t'y_0+(1-t')y_1)\in C(Z),\\
			&\ \text{ where }t'= \textstyle{\frac{t_2-t_1}{1-t_1}}.
		\end{align*}
		Combining all three cases, we conclude that
		$(x,y)\in C(Z)$ and
		consequently $C(X)\times C(Y)\subseteq C(Z)$.
	\end{proof}

	\subsection{Stasheff complexes vs Collapsed Multiplihedra}\label{StasMulti}
	We shall use the Definition \ref{def:Associahedra1}  for Stasheff complexes. Similarly, due to Lemma \ref{lem:ConRealMulti}, we will use Definition \ref{def:CollMulti} for collapsed multiplihedra.
	
	\begin{prop}\label{thm:StasMulti}
		Stasheff complexes $K_{n+1}$ and collapsed multiplihedra $J'_n$ are combinatorially isomorphic.
	\end{prop}
	\begin{proof}
		Both $K_{n+1}$ and $J'_n$ are cellular  complexes whose face posets are isomorphic to $\mathfrak{P}(n+1)$ and $\mathfrak{J}'_n$ respectively. Therefore, in order to exhibit an isomorphism between $J'_n$ and $K_{n+1}$, it suffices to find a bijection between $\mathfrak{P}(n+1)$ and $\mathfrak{J}'_n$ as posets.
		
		Define $\phi:\mathfrak{J}'_n\to \mathfrak{P}(n+1)$ as follows 
		\begin{align*}
			f(X_1) & \mapsto f(X_1)a_{n+1}:=(X_1)a_{n+1}\\
			f((X_1)\cdot\ldots \cdot (X_{k-1})\cdot (X_k)) & \mapsto ((X_1)\cdot \ldots \cdot (X_{k-1})\cdot (X_k)) a_{n+1} := (X_1)\ldots (X_{k-1}) (X_k)a_{n+1}
		\end{align*}
		
		\begin{align*}
			\phi\left(f(X_1)\ldots f(X_{k-1})f(X_k)\right)
			&= f(X_1)\ldots f(X_{k-1}) f(X_k)a_{n+1}\\
			&= f(X_1) \ldots f(X_{k-1}) ((X_k)a_{n+1})\\
			&= f(X_1) \ldots f(X_{k-2})((X_{k-1})((X_k)a_{n+1}))\\
			&= \cdots\\
			&= (X_1)(\ldots ((X_{k-1})((X_k)a_{n+1})) \ldots),\\[4pt]
			\phi\left(f((X_1)\cdot (X_2))f((X_{3})\cdot (X_4)\cdot (X_5))\right)
			&= f((X_1)\cdot (X_2))(((X_{3})\cdot (X_4)\cdot (X_5))a_{n+1})\\
			&= ((X_1)\cdot (X_2))((X_{3}) (X_4) (X_5)a_{n+1})\\
			&= (X_1)(X_2)((X_{3}) (X_4)(X_5)a_{n+1}).
		\end{align*}
		Here $X_i$'s are some rule of association of the elements $a_1,a_2,...,a_n$ in $A$ of some length such that the total length of all $X_i$'s is $n$ and $a_{n+1}$ is some different element in $A$.
		In the above correspondence, note that the bracketing in $X_i$'s are not changed. We only include some pair of brackets removing $f$'s or remove $\cdot$ and keep it as it is with an extra letter $a_{n+1}$ on the right to get a bracketing of the word $a_1a_2\ldots a_{n+1}$. Also, note that each parenthesis right to the letter $a_{n+1}$ determines the number of $f$ and their position as well, where no parentheses mean only single $f$ with the $\cdot$'s in between the associated words.
		Thus, the position of each $f$ and $\cdot$ gives a unique bracketing of the word $a_1a_2\ldots a_{n+1}$ and the process can also be reversed. So $\phi$ is bijective. 
		Now in order to check $\phi$ preserves the poset relation, we need to show $\phi(P\prec P')\implies \phi(P)< \phi(P').$ There are three possible ways (cf. operation (\ref{it:op1}), (\ref{it:op2}), (\ref{it:op3})) by which $P$ can be related to $P'$.
		\begin{enumerate}
			\item $P$ is obtained from $P'$ by adding  brackets in domain. Since $\phi$ do not interact with the brackets in domain, $\phi(P)$ is also obtained from $\phi(P')$ by adding brackets i.e., $\phi(P)< \phi(P')$.
			\item $P$ is obtained from $P'$ by replacing $\cdot$ by `$)f($'. Thus $P$ contains more $f$ than $P'$. But from the correspondence, we know each $f$ corresponds to a pair of brackets, so $\phi(P)$ must be obtained from $\phi(P')$ by adding brackets i.e., $\phi(P)< \phi(P')$.
			\item $P$ is obtained from $P'$ by removing one or more consecutive $\cdot$ by adding a pair of brackets that encloses all the adjacent elements to those $\cdot$. To obtain $P$, this process adds brackets to $P'$ and $\phi$ does not change the parent bracketing. So so $\phi(P)$ must be obtained from $\phi(P')$ by adding brackets i.e., $\phi(P)< \phi(P')$.
		\end{enumerate}
		Thus $\phi$ defines a bijection of the posets $\mathfrak{J}'_n$ and $\mathfrak{P}(n+1)$. Hence $J'_n$ and $K_{n+1}$ are combinatorially isomorphic.
	\end{proof}

	\subsection{Collapsed Multiplihedra vs Graph Cubeahedra}\label{MultiCubea}
	\begin{prop}\label{thm:MultiCubea}
		Collapsed multiplihedra $J'_{n+1}$ and graph cubeahedra $\mathcal{C}P_n$ for path graph $P_n$ with $n$ nodes are combinatorially isomorphic.
	\end{prop}
	\begin{proof}
		Recall from Theorem \ref{thm:OrderCubeahedra} that the graph cubeahedron $\mathcal{C}P_n$ is a convex polytope of dimension $n$ whose face poset is isomorphic to the set of design tubings of $P_n$. Recall that the collapsed multiplihedra $J'_{n+1}$ is a cellular complex of dimension $n$ whose face poset is isomorphic to $\mathfrak{J}'_{n+1}$. Thus, to describe an isomorphism, it is enough to prove a bijection at the poset level.\\
		\hspace*{0.5cm}A bijection between the design tubings and the elements of $\mathfrak{J}'_{n+1}$ is defined through the following correspondences:
		\begin{itemize}
			\item Put $a_1$ through $a_{n+1}$ starting from the left of the left-most node to the right of the right-most node of the graph:
			\begin{figure}[H]
				\centering
				\includegraphics[page=1]{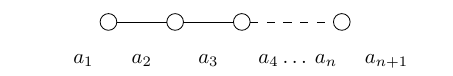}
				\caption{Initial step}
			\end{figure}
			\item Each round tube corresponds to a pair of parentheses. If the round tube includes $k$-th and $(k+r-1)$-th node of the graph, then the corresponding parentheses include $a_k$ through $a_{k+r}$.
			\begin{figure}[H]
				\centering
				\includegraphics[page=3]{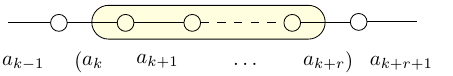}
				\caption{Correspondence of round tube}
			\end{figure}
			\item Each square tube corresponds to the inclusion of `$)f($' in the string $f(a_1a_2\ldots a_{n+1})$. If the square tube include $k$-th node of the graph, then `$)f($' will be included in between $a_k$ and $a_{k+1}$.
			\begin{figure}[H]
				\centering
				\includegraphics[page=5]{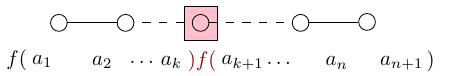}
				\caption{Correspondence of square tube}
			\end{figure}
			\item An empty node in a tubing corresponds to `$\cdot$' i.e. if $k$-th node of the graph is not included by any tube of the given tubing, then put a `$\cdot$' between $a_k$ and $a_{k+1}$.  
			\begin{figure}[H]
				\centering
				\includegraphics[page=6]{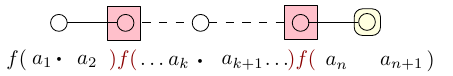}
				\caption{Correspondence of empty node}
			\end{figure}
		\end{itemize}
		Finally, as the position of each tube and its appearance give a unique element of $\mathfrak{J}'_{n+1}$, we get a bijective correspondence between design tubings and elements of $\mathfrak{J}'_{n+1}.$ An example, assuming $n=6$, is given below.
		\begin{figure}[H]
			\centering
			\begin{subfigure}[b]{0.47\linewidth}
				\centering
				\includegraphics[page=7, width= \linewidth]{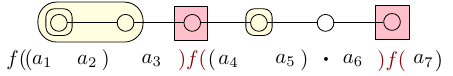}
			\end{subfigure}
			\hspace{0.2cm}
			\begin{subfigure}[b]{0.47\linewidth}
				\centering
				\includegraphics[page=8, width= \linewidth]{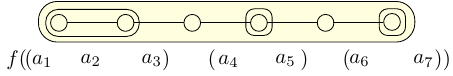}
			\end{subfigure}
			\caption{Bijection between design tubings and multiplihedra}
		\end{figure}
		It follows from the correspondence that the removal of a round tube corresponds to the removal of a pair of parentheses or adding `$\cdot$' and the removal of a square tube corresponds to replacing `$)f($' by `$\cdot$'. This shows that the poset relation between design tubings matches with the poset relation in $\mathfrak{J}'_{n+1}.$ As the two posets are isomorphic, this finishes the proof.
	\end{proof}

	\bibliographystyle{siam}
	
	\bibliography{Bibliography}

\end{document}